\documentclass[11pt]{amsart}
\usepackage{newlfont,amsmath,enumerate,amssymb,graphicx,float}
\usepackage[labelformat=empty]{caption}
\newcommand{\bbC}{{\mathbb{C}}}

\newcommand{\bbQ}{{\mathbb{Q}}}
\newcommand{\bbR}{{\mathbb{R}}}
\newcommand{\bbZ}{{\mathbb{Z}}}

\newcommand{\PGL}{{\mathrm{PGL}}}
\newcommand{\SL}{{\mathrm{SL}}}
\newcommand{\GL}{{\mathrm{GL}}}

\newcommand{\bfe}{\mathbf{e}}

\newcommand{\bfr}{\mathbf{r}}
\newcommand{\bfs}{\mathbf{s}}
\newcommand{\bfu}{\mathbf{u}}
\newcommand{\bfv}{\mathbf{v}}
\newcommand{\bfw}{\mathbf{w}}

\newtheorem{thm}{Theorem}[section]
\newtheorem{deff}[equation]{Definition}
\newtheorem{lemma}[thm]{Lemma}

\newtheorem{prop}[thm]{Proposition}

\newtheorem*{main}{Main Theorem}
\newtheorem{cor}[thm]{Corollary}
\theoremstyle{remark}

\def\Z{{\mathbb Z}}
\def\H{{\mathbb H}}

\begin{document}


\title[integral forms]{Geometry of integral binary hermitian forms}

\author{Mladen Bestvina}

\address{Mladen Bestvina, Department of
 Mathematics, University of Utah, Salt Lake City, UT 84112}
\email{bestvina@math.utah.edu}

\author{Gordan Savin}

\address{Gordan Savin, Department of
 Mathematics, University of Utah, Salt Lake City, UT 84112}
\email{savin@math.utah.edu}

\thanks{Both authors gratefully acknowledge the support by the
  National Science Foundation}

\begin{abstract}
We generalize Conway's approach to integral binary quadratic forms to
study integral binary hermitian forms over quadratic imaginary
extensions of $\bbQ$.  We show that every indefinite anisotropic form determines a plane (``ocean'') in
Mendoza's spine associated with the corresponding Bianchi group in the hyperbolic 3-space.  The ocean 
can be used to compute the group of integral transformations preserving the hermitian form. 
\end{abstract}
\maketitle

\section{Introduction}  

Serre observed that there is a 1-dimensional, $\SL_2(\bbZ)$-invariant,
deformation retract $T$ of the hyperbolic plane $\mathbb H^2$. The
retract $T$ is a 3-valent tree, and connected components of $\mathbb
H^2 \setminus T$ correspond to cusps.  Conway \cite{Co}
uses the tree $T$ to study integral binary quadratic forms. More precisely, a cusp 
 $\alpha\in \mathbb P^1(\bbQ)$ corresponds to a pair $(m,n)$ of relatively
prime integers, unique up to a sign, such that $\alpha=m/n$. In particular, if $f$ is an
integral, binary quadratic form, then the integer $F(\alpha)=f(m,n)$
is well defined. Thus the form $f$ assigns values to the connected
components of the complement of $T$.  
If $f$ is indefinite and anisotropic, then Conway defines a
river $R\subseteq T$.  It is a tree geodesic (i.e. a line) in $T$ whose edges belong to the
closures of two connected components on which the form has opposite
signs. The river $R$ can
be used to estimate and find the minimum of $|f|$, as well as to compute
the group $SO(f)$ of special integral transformations preserving $f$ (see Section \ref{S1}).

Let $k=\bbQ(\sqrt{D})$ be a quadratic imaginary extension of $\bbQ$,
where $D<0$ is a fundamental discriminant. Then
$A=\bbZ[\frac{D+\sqrt{D}}2]$ is the ring of algebraic integers in $k$. 
 Mendoza \cite{Me} and Ash \cite{As} have defined a 2-dimensional,
$\GL_2(A)$-invariant deformation retract $X$ of the hyperbolic space
$\mathbb H^3$ such that the connected components of $\mathbb H^3\setminus X$ correspond 
to cusps for the group $\GL_2(A)$. 
Using the natural Euclidean metric on horospheres in
$\mathbb H^3$, the spine $X$ is a CAT(0) complex (see \cite{iain})
whose 2-cells are Euclidean polygons.  For example, if $A$ is a
Euclidean domain then (and only then) every 2-cell is isometric to a
fundamental domain for $A$ in $\bbC$ (the Voronoi cell).  In
particular, if $A$ is the ring of Eisenstein integers ($D=-3$),
then every 2-cell is a regular hexagon, and if $A$ is the ring of
Gaussian integers ($D=-4$), then every 2-cell is a square.

Let $V$ be a 2-dimensional vector space over $k$ and let 
$f$ be a hermitian form on $V$. We say that $f$ is integral if 
there exists an $A$-lattice $M$ in $V$ (i.e. M is a rank 2
$A$-module, discrete in $V\otimes_k\bbC$)
such that the restriction of $f$ to $M$ has values in $\bbZ$. 
Since the connected components of the complement of $X$ in $\mathbb H^3$ 
correspond to  cusps, the form $f$ assigns values
to these connected components.  Let $v$ be a vertex of $X$. As our first result, 
we show that $f$ is determined by the values at the connected components that contain $v$ 
in the closure. 
 If $f$ is indefinite and anisotropic, then we 
define the ocean $C$ as the union of the 2-cells of $X$ that belong to
the closures of two components on which $f$ has opposite signs. A simple topological argument
shows that the ocean is non-empty. We then show that the values of $f$ along the ocean are bounded by 
an explicit constant depending on the discriminant of $f$. 
This immediately implies 
an old result, due to Bianchi \cite{Bi}, that there are only finitely many
 $\GL_2(A)$-equivalence classes of integral binary hermitian forms of a fixed discriminant.   

Let $U(f)$ be the subgroup of $\GL_2(A)$ preserving the
form $f$.  It is a Fuchsian group. It acts, with a compact quotient,
on a hyperplane $\mathbb H^2_f\subseteq \mathbb H^3$ whose boundary is
the circle defined by $f=0$. We show that the ocean $C$, corresponding
to the form $f$, is topologically a plane on which $U(f)$ acts with
a compact quotient. In particular, $C/\Gamma$ is a nonpositively curved
(i.e. locally CAT(0)) classifying space for any torsion-free subgroup $\Gamma$ of $U(f)$.
Moreover, once the combinatorial structure of
$C$ is known, it is not difficult to find generators of $U(f)$. Thus we have a method of finding 
generators and relations of Fuchsian groups different than the standard approach based on 
ideas of Ford \cite{Fo} (see \cite{Vt} for recent developments). 

We also have very specific results in the case $D=-3$.  
In this case the combinatorial structure of the spine $X$ has an 
arithmetic description that we learned from Marty Weissman.  Here  
many elementary features of Conway's treatment carry over, and  
we prove two sharp results on the minima of integral 
binary hermitian forms (Theorems \ref{T2} and \ref{T3}). 

To summarize, the main result of the paper is the following.

\begin{main}
Let $f$ be an indefinite and anisotropic  integral binary hermitian form 
 over the imaginary quadratic field with discriminant
$D<0$. Then the ocean $C$ is homeomorphic to the plane
and $U(f)$ acts on it
cocompactly. Moreover, the nearest point projection from $C$ to the
hyperbolic plane $\mathbb H^2_f$ obtained by taking the convex hull of the circle
$f=0$ is a homeomorphism. A tiling of $\mathbb H^2_f$ by hyperbolic polygons is obtained by projecting the cells of the ocean $C$. 
\end{main}

The paper is organized as follows. In Sections 2-4 we review some
background material: Conway's topograph, arithmetic of hermitian
forms, and Mendoza's spine in the hyperbolic 3-space $\mathbb H^3$
associated to a Bianchi group $\PGL_2(A)$. In
Section 5 we define the ocean associated to an indefinite and anisotropic integral
binary hermitian form $f$, and in Sections
6-7 we prove that the ocean is a contractible surface and that it
maps homeomorphically to the hyperbolic plane $\mathbb H^2_f$ under the
nearest point projection. In Section 8 we present an algorithm for
computing the symmetry group $PU(f)$ of the form $f$. We end the paper
with several explicit computations  over the Gaussian and
Eisenstein integers. 

We would like to thank Marty Weissman (his idea to study 
cubic binary forms over Eisenstein integers using the spine  
was the starting point of our investigations),  
 Michael Wijaya  for  pointing out a number of inaccuracies and the referee for careful 
reading of the paper and numerous suggestions. 

\section{Conway's tree} \label{S1}  

In this section we describe briefly Conway's approach to quadratic forms
(see also Hatcher's book \cite{hatcher:topographs}). 
We say that a  binary quadratic form $ax^2+2hxy+ cy^2$ is integral if it takes integral 
values on the lattice $\bbZ^2$. This means that $a$ and $c$ are 
integers while $h$ is a half-integer. The integer $D=4h^2-4ac$ is called 
the {\em discriminant} of the form. The form is indefinite if and only if $D>0$ and it is anisotropic 
if and only if $D$ is not a square. 

\smallskip 
A vector $\bfv=(m,n)$ in $\bbZ^2$ 
is called primitive if $m$ and $n$ are relatively prime integers. 
A {\em lax} vector is a pair $\pm \bfv$ where 
$\bfv$ is a primitive vector. In order to keep notation simple we shall 
omit $\pm$ in front of $\bfv$. A lax basis is a set $\{\bfu,\bfv\}$ of two 
lax vectors such that $(\bfu,\bfv)$ is a basis of the lattice $\bbZ^2$ for 
one and therefore for any choice of signs in front of $\bfu$ and $\bfv$. 
A lax superbasis is a set $\{\bfu,\bfv,\bfw\}$ of three lax vectors such that 
any two form a lax basis. Notice that this is equivalent to 
$\bfu+\bfv+\bfw=0$ for some choice of signs.

Conway defines a 2-dimensional ``topograph'' as follows: $2$-dimensional 
regions correspond to lax vectors, edges to lax bases and vertices 
to lax superbases, preserving incidence relations. For example, 
the edge corresponding to the lax basis $\{\bfu,\bfv\}$ is shared by 
the regions corresponding to the lax vectors $\bfu$ and $\bfv$, and so on. 
Let $\bfe_1$ and $\bfe_2$ denote the standard basis vectors in $\bbZ^2$.  
A  part of this topograph is given in the following figure.

\begin{picture}(300,200)(-100,-20)

\put(130,120){\line(0,1){25}}
\put(70,120){\line(0,1){25}}
\put(100,100){\line(0,-1){36}}
\put(100,100){\line(3,2){30}}
\put(100,100){\line(-3,2){30}}
\put(70,120){\line(-3,-2){20}}
\put(130,120){\line(3,-2){20}}
\put(100,64){\line(-3,-2){20}}
\put(100,64){\line(3,-2){20}}

\put(85,80){$\bfe_1$} 
\put(110,80){$\bfe_2$} 

\put(85,120){$\bfe_1+\bfe_2$} 
\put(85,40){$\bfe_1-\bfe_2$} 

\put(25,125){$2\bfe_1+\bfe_2$} 
\put(140,125){$\bfe_1+2\bfe_2$}


\end{picture}

Edges and vertices form a 3-valent tree, denoted by $T$.  We can
assign values to the complementary regions of the topograph by
evaluating a quadratic form $f$ on lax vectors. For example, evaluating
the quadratic form $x^2+xy-y^2$ on lax vectors in the above figure
gives

\begin{picture}(300,200)(-100,-20)

\put(130,120){\line(0,1){25}}
\put(70,120){\line(0,1){25}}
\put(100,100){\line(0,-1){36}}
\put(100,100){\line(3,2){30}}
\put(100,100){\line(-3,2){30}}
\put(70,120){\line(-3,-2){20}}
\put(130,120){\line(3,-2){20}}
\put(100,64){\line(-3,-2){20}}
\put(100,64){\line(3,-2){20}}

\begin{thicklines} 
\put(130,120){\line(0,1){25}}
\put(100,100){\line(0,-1){36}}
\put(100,100){\line(3,2){30}}
\put(100,64){\line(-3,-2){20}}
\end{thicklines}

\put(85,80){$1$}
\put(110,80){$-1$}

\put(96,120){$1$}
\put(94,40){$-1$}

\put(55,125){$5$}
\put(140,125){$-1$}

\put(100,64){\circle*{4}}
\put(130,120){\circle*{4}}

\end{picture}

For a vertex $v$ define $inv(v)$ as the sum of the values of $f$ on
the three incident regions, and for an edge $e$ let $inv(e)$ denote the
sum of the values of $f$ on the two incident regions. Then the
parallelogram law
$$f(\bfe_1+\bfe_2)+f(\bfe_1-\bfe_2)=2~f(\bfe_1)+2~f(\bfe_2)$$ can be
restated conveniently as
$$inv(v)+inv(v')=4~inv(e)$$
whenever vertices $v,v'$ are connected by an edge $e$. 
It is now easy to fill in the values of $f$ around the topograph.
In particular, the form $f$ is determined by its values on a lax superbasis. 

Now assume that the form $f$ is indefinite and anisotropic. 
 Let $R\subseteq T$ be the union of all edges that
belong to the closures of two regions on which the form has opposite signs. 
Conway shows that $R$ is a tree geodesic (i.e. a line) and therefore
calls it a river.  

Some classical results on quadratic forms can be easily proved using
the tree.  For example, a well known result (Theorem 35 in \cite{Si})
says that if $f$ is an indefinite form of discriminant $D$ then there
exists a non-zero element $\bfv$ in $\bbZ^2$ such that
$|f(\bfv)|\leq\sqrt{\frac{D}{5}}$.  To see this, assume that $f$ takes
values $a$, $b$ and $c$ at a vertex of the tree.  Then one easily
checks that the discriminant of $f$ is
\[ 
D= a^2 + b^2+ c^2 -2ab-2bc -2ac. 
\] 
Furthermore, we can assume that $a,b>0$ and $c<0$ by picking the vertex to 
be on the river. Then $-2ab$ is the only negative term in the above 
expression for $D$. Since $a^2+b^2-2ab=(a-b)^2\geq 0$ we have 
\[ 
D\geq c^2-2bc-2ac \geq 5 s^2 
\] 
where $s$ is the minimum of $a, b$ and $-c$.

We shall now explain how the river can be used to compute the 
subgroup $SO(f)\subseteq \SL_2(\bbZ)$ 
preserving the form $f$.  The key observation is that
$SO(f)$ acts on the river, and that the values of $f$ along the edges of $R$ must be 
bounded. Indeed, if $a<0<b$ are the values of $f$ on the lax vectors 
corresponding to an edge in the river, then $D=4h^2-4ab\geq 4|ab|$ which implies 
that $|a|,|b|\leq D/4$. 
Thus the values of $f$ along $R$ are periodic.  
 Let $v$ and $v'$ be two vertices on $R$. Since $\SL_2(\mathbb Z)$ 
acts transitively on the vertices of $T$, there is 
$g\in \SL_2(\mathbb Z)$ such that $g(v)=v'$.  If $f(g(\bfv))=f(\bfv)$ for all three 
lax vectors in the superbasis corresponding to $v$, then $g\in SO(f)$. It follows that 
$SO(f)\cong \mathbb Z$ where a generator of $SO(f)$ acts on $R$ by translating by the period.

To motivate our arguments in Section \ref{s:main}, here is a brief
outline of a proof that the river $R$ is a line. We now view $T$ as a
subset of the hyperbolic plane $\mathbb H^2$. In the upper half plane
realization of $\mathbb H^2$, the spine $T$ is the $\SL_2(\bbZ)$-orbit
of the bottom arc of the standard fundamental domain for
$\SL_2(\mathbb Z)$.  One shows:

\begin{itemize}
\item $R$ is a 1-manifold. Indeed, at every vertex $v\in R$ there are three
  complementary regions, two of them have one sign, and the third has
  the other sign, so exactly two of the three edges incident to $v$
  belong to $R$. 
\item $SO(f)$ acts cocompactly on $R$. Since the values of $f$ along $R$ are bounded,  
 there are finitely many $SO(f)$-orbits of vertices in $R$, so $SO(f)$ acts cocompactly on $R$. 
\item $R$ consists of finitely many lines. Since $SO(f)$ acts discretely on $\mathbb H^1_f\subseteq \mathbb H^2$, 
the geodesic line connecting the two points on the circle at infinity where $f=0$, it follows that 
  $SO(f)\cong \Z$. Since $SO(f)$ acts cocompactly on $R$, it consists of finitely many lines.
\end{itemize}

It remains to show that $R$ is connected.
The nearest point projection $\pi:R\to\H^1_f$ is a proper map
  (since it is $SO(f)$-equivariant and $SO(f)$ acts on both spaces
  cocompactly). It follows that $R$ is contained in a Hausdorff
  neighborhood of $\H^1_f$. Choose a finite index subgroup $\Gamma\subset SO(f)$ that
  preserves each line in $R$. 
   Then the quotient $S=\H^2/\Gamma$ is an annulus,
  $\H^1_f/\Gamma$ is a geodesic circle in $S$, and for each component
  $L$ of $R$, $L/\Gamma$ is a circle in $S$ homotopic to
  $\H^1_f/\Gamma$. But there is a line in $S$ that intersects
  $R/\Gamma$ transversally in one point (obtained by projecting the
  union of two rays in $\H^2$ joining an interior point of an edge in
  $R$ with the corresponding two cusps), so $R$ is connected.

The reader may also show that the nearest point projection $\pi:R\to
\H^1_f$ is a homeomorphism. In view of properness of $\pi$ this is
equivalent to $\pi$ being locally injective. We will verify in Section
\ref{s:main} the analogous statement in one dimension higher.

\section{Arithmetic}

In this section we fix some basic notation and terminology.  Let $k$
be a quadratic imaginary extension of $\bbQ$ and $A\subseteq k$ the
maximal order (i.e. the ring of integers). Then every element $\alpha$
in $k$ satisfies
\[
x^{2}-tr(\alpha)x + N(\alpha)=0
\]
where $tr(\alpha)$ and $N(\alpha)$ are the trace and the norm of $\alpha$. 
Let $D<0$ be the discriminant of $A$. It is the determinant of the trace pairing 
on $A$, considered as a free $\bbZ$-module of rank 2. 
Then $k=\bbQ(\sqrt{D})$ and 
$A=\bbZ[\tau]$ where $\tau=\frac{D+\sqrt{D}}{2}$. The order $A$ is a 
Euclidean domain if and only if $D=-3,-4,-7, -8$ or $ -11$.  

\smallskip 

Let $V$ be a 2-dimensional space $V$ over $k$. A function 
 $f: V \rightarrow \mathbb Q$ is a (binary) hermitian form over $k$ if,
 after fixing a basis $( \bfu,\bfv )$ of $V$, it can be expressed
as
 \[
 f(x,y) = 
 \left(\begin{array}{cc} 
 x & y \\
 \end{array}\right) 
 \left(\begin{array}{cc} 
 a & \nu  \\
 \bar{\nu} & c
 \end{array}\right)
  \left(\begin{array}{c} 
 \bar{x} \\
 \bar{y} \\
 \end{array}\right) 
 \]
 where $a$ and $c$ are in $\bbQ$ and $\nu$ is in $k$.

 \begin{prop} \label{hermitian}
 A hermitian form $f$ is determined by values at four
 non-zero vectors provided that, viewed as elements in the projective
 line $\mathbb P^1(\bbC)$, they do not lie on a line or a circle.
 \end{prop}
 \begin{proof} Assume that the form $f$ is represented by the matrix 
 $\left(\begin{smallmatrix}a&\bar\nu\\
\nu&c\end{smallmatrix}\right)$.
 Using the action of $\PGL_{2}(k)$ we can assume that the four vectors are 
 \[
 \left(\begin{array}{c} 
 1 \\
 0 \end{array}\right),
 \left(\begin{array}{c} 
 0 \\
 1 \end{array}\right), 
 \left(\begin{array}{c} 
 \alpha \\
 1 \end{array}\right), 
 \text{ and }
 \left(\begin{array}{c} 
 \beta \\
 1 \end{array}\right).  
 \]
 The values of $f$ on these four vectors are $a$, $c$,
 $aN(\alpha)+c+tr(\alpha\nu)$ and $aN(\beta)+ c +tr(\beta\nu)$. Then
 one can solve for
 $\nu,\overline\nu$ provided that
 $\det\left(\begin{smallmatrix}\alpha&\overline\alpha\\
\beta&\overline\beta\end{smallmatrix}\right)\neq 0$. The condition then is
that $\alpha\overline\beta\neq \overline\alpha\beta$, i.e. that $\frac
{\alpha}{\beta}$ is not real, which is equivalent to saying that
$0,\alpha,\beta,\infty$ are not on the same line.
 \end{proof} 
 
 Let $f$ be a hermitian form on $V$. Let $(\bfu,\bfv)$ be a basis of $V$, and 
 $\left(\begin{smallmatrix}a&\nu\\
\overline\nu&c\end{smallmatrix}\right)$ the matrix of $f$ in this basis. The rational number 
\[
 \Delta=D(ac-N(\nu))
 \]
 depends on the choice of the basis, but its class in $\bbQ^{\times}/N(k^{\times})$ does not. 
  Note that $\Delta<0$ for definite forms, while $\Delta>0$ for indefinite forms, 
  by our normalization.  
  The number $\Delta$ (rather, its class in $\bbQ^{\times}/N(k^{\times})$)
  is called the discriminant of the form $f$. This notion can be refined for integral forms. 
  More precisely, fix an $A$-lattice $M$ in $V$. We say that the form $f$ is {\em integral}
  if it takes integral values on the lattice $M$. If $M$ is spanned by $(\bfu,\bfv )$ then 
  $a$ and $c$ are integers, while $tr(\nu\cdot \gamma)\in \mathbb Z$ for all $\gamma\in A$, i.e. 
  $\nu\in A^*$, the dual lattice. Recall that $[A^*:A]=|D|$.  It follows that 
  $|D|\cdot N(\nu)=[A^*:(\nu)]$ is an integer for every $\nu$ in $A^*$. Thus  the number 
 $\Delta=D(ac-N(\nu))$ is an integer which is independent of the choice of a basis in $M$. 
   
  Given $\Delta \in \bbQ^{\times}$ let $\left(\frac{D,\Delta}{\bbQ}\right)$ be  the quaternion algebra over $\bbQ$ 
  consisting of all pairs $(x,y)\in k^2$ where 
  addition is defined coordinate-wise and multiplication by 
  \[ 
  (x,y)\cdot (z,w) =(xz+\Delta w\bar{y}, zy+\bar{x}w).
  \] 
  The algebra $\left(\frac{D,\Delta}{\bbQ}\right)$ has the norm $N(x,y)=N(x)-\Delta N(y)$. We have the following 
  (see \cite{MR}). 
  
   \begin{prop} Let $f$ be a binary hermitian form over $k$ with the discriminant $\Delta$. 
 If $\Delta <0$, i.e. $f$ is definite, assume that $f$ 
 is positive definite. Then $f$ is equivalent to the
 norm form of the quaternion algebra
 $\left(\frac{D,\Delta}{\bbQ}\right)$.
   \end{prop} 
 \begin{proof} 
 We can view $f$ as a rational quadratic form in 4 variables. 
 By our assumption the quintic form 
 $f-x_5^2$ is indefinite. By the Hasse-Minkowski Theorem such a form represents 
 $0$. In particular, there 
 exists an element $\bfu\in V$ such that $f(\bfu)=1$. 
Thus, we can find a basis $(\bfu,\bfv)$ in $V$ such that the matrix of $f$ is 
$\left(\begin{smallmatrix}1&\nu\\
\overline\nu&c\end{smallmatrix}\right)$. 
  Then 
 \[
 f(x,y)=N(x+\nu y) + (c-N(\nu))N(y).    
 \]
 Let $z$ be in $k$ such that $N(z)=-D$. Then $f(x',y')=N(x')-\Delta N(y')$ where 
 $x'=x+\nu y$ and $y'=y/z$. The proposition is proved.  
 \end{proof}
 
 Note that the form $f$ is anisotropic if and only if 
 $\left(\frac{D,\Delta}{\bbQ}\right)$ is a division algebra, and this holds if and only 
 if  the Hilbert symbol  $(D, \Delta)_p$ is non-trivial for at least one prime $p$.

\section{Hyperbolic space and Spine} \label{Spine}

In this section we define Mendoza's $\GL_2(A)$-equivariant 
retract of $\mathbb H^3$ \cite{Me} and establish some of its basic properties. 

Let $V$ be the space of real $2\times 2$ hermitian symmetric matrices
and $P\subseteq V$ the cone of positive definite matrices. The subset
of $P$ defined by the equation $\det(x)=1$, where $x$ is in $P$, is a
realization of the 3-dimensional hyperbolic space $\mathbb H^3$. Note
that the boundary points of $\mathbb H^3$ correspond to
rays on the boundary of the cone $P$.

 The group $\GL_2(\bbC)$ acts on $V$ by $g(x)=gxg^{\ast}$, where $g\in \GL_2(\bbC)$, 
$x\in V$ and $g^{\ast}$ is the conjugate-transpose of $g$.  Let $G$ be the subgroup of $\GL_2(\bbC)$ 
consisting of all  $g$ such that $\det(g)$ is a complex number of norm 1. 
The group $G$ preserves the quadratic form $\det(x)$ on $V$. In 
particular $G$ acts naturally on $\mathbb H^3\subseteq P$. 
Note that $\GL_2(A)$ is contained in $G$.

Let $\mathbb P(k)=k\cup\{\infty\}$ be the projective line over $k$.
Let $I\subseteq A$ be a non-zero ideal. Let $\mathbb P(I)$ be the set
of equivalence classes of pairs $(a,b)\in I^{2}$ (the equivalence
relation is defined by multiplication by $A^{\times}$) such that $a$
and $b$ generate the ideal $I$. We have an obvious injection of
$\mathbb P(I)$ into $\mathbb P(k)$, and the image is independent of
the choice of $I$ in an ideal class. Let $I_{1}, \ldots ,
I_{h}\subseteq A$ be representatives of all ideal classes. We have a
disjoint union (into $\GL_{2}(A)$-orbits)
\[
\mathbb P(I_{1}) \cup \ldots \cup \mathbb P(I_{h})= \mathbb P(k). 
\]
Indeed, if $\alpha=\frac{a}{b}\in k$, where $a,b\in A$, 
let $I$ be the ideal generated by $a$ and $b$. The ideal class of $I$ is obviously independent 
of the choice of $a$ and $b$. Then $\alpha$ is in the image of  
$(a,b)\in \mathbb P(I)$.
Bianchi \cite{Bi} has shown that each $\mathbb P(I)$ is one $\GL_{2}(A)$-orbit.  
Elements in $\mathbb P(k)$ are also called cusps for $\GL_{2}(A)$. 
There is a natural $\GL_2(A)$-equivariant map
$\alpha\rightarrow x_{\alpha}$  from $\mathbb P(k)$ to the 
boundary of the cone $P$ defined by 
\[
x_{\alpha}= \frac{1}{N(I)} 
 \left(\begin{array}{cc} 
 N(a) & a\bar b \\
 \bar a b & N(b)
 \end{array}\right)
\] 
where $N(I)$ is the norm of the ideal generated by $a$ and $b$. 

Let $\langle x,y \rangle$ be the symmetric bilinear form on $V$ such that 
$\langle x, x\rangle =\det(x)$. Define the distance from the cusp $\alpha$ to 
$w\in \mathbb H^3$ by $d_{\alpha}(w)=\langle x_{\alpha}, w\rangle$.  Note that 
$d_{g(\alpha)}(g(w))=d_{\alpha}(w)$ for any $g$ in $\GL_2(A)$ 
since $\langle x,y \rangle$ is $\GL_2(A)$-invariant. 

Let $w=(z,\zeta)\in \bbC \times \bbR^{+}$ in the upper-half space model of the hyperbolic space. 
The distance $d_{\alpha}(w)$ is given by 
\[
d_{\alpha}(w)=\frac{N(b)}{N(I)} \cdot \frac{|z-\alpha|^{2}+\zeta^{2}}{\zeta }
\]
if $\alpha\neq \infty$ and $d_{\infty}(w)=1/\zeta$. (Geometrically,
this represents the exponential of the signed hyperbolic distance between $w$
and a suitable horosphere centered at $\alpha$. This formula is due to
Siegel \cite{Si2}.)
Let $t>0$. The set 
\[ 
B_{\alpha}(t)= \{ w\in \mathbb H^3 | d_{\alpha}(w)\leq t\} 
\] 
is called a horoball. If $\alpha\neq \infty$ then $B_{\alpha}(t)$ is a 
Euclidean 3 dimensional ball 
of radius $tN(I)/2N(b)$  touching the boundary of $\mathbb H^3$ at $\alpha$. 

 If we fix $t$ and an ideal $I$, then $\GL_{2}(A)$  acts transitively on 
horoballs for $\alpha\in \mathbb P(I)$.  For every cusp $\alpha$ define 
\[
H_{\alpha}= \{ w\in \mathbb H^{3} ~|~ d_{\alpha}(w)\leq d_{\beta}(w) 
\text{ for all } \beta \in \mathbb P(k)\}. 
\]
The spine $X$ is defined  by 
\[
X=\cup_{\alpha\neq \beta} H_{\alpha}\cap H_{\beta}
\]
where the union is taken over all distinct pairs of cusps.

Spine is a two dimensional cell-complex.  Its 2-cells consist of
points equidistant to two cusps, with no other cusps being
closer. Since all points in $\mathbb H^3$ equidistant to two cusps
form a hyperbolic subspace $\mathbb H^2$ (a Euclidean hemisphere
centered at a point in $\bbC$, or a Euclidean half-plane vertical to
$\bbC$), each cell is a convex polygon in $\mathbb H^2$.  The radial
projection from $\alpha$ is a homeomorphism between the boundary of
$B_{\alpha}(t)$ and the boundary of $H_{\alpha}$. Thus the Euclidean
metric on the boundary of $B_{\alpha}(t)$ induces a Euclidean metric
on the boundary of $H_{\alpha}$. In this way $X$ becomes a CAT(0)
complex. See \cite{iain}, based on the work of Rivin \cite{rivin}.

\setlength\fboxsep{0pt}
\setlength\fboxrule{1pt}

\begin{figure}[H]
\begin{center}
\fbox{\includegraphics[scale=0.25]{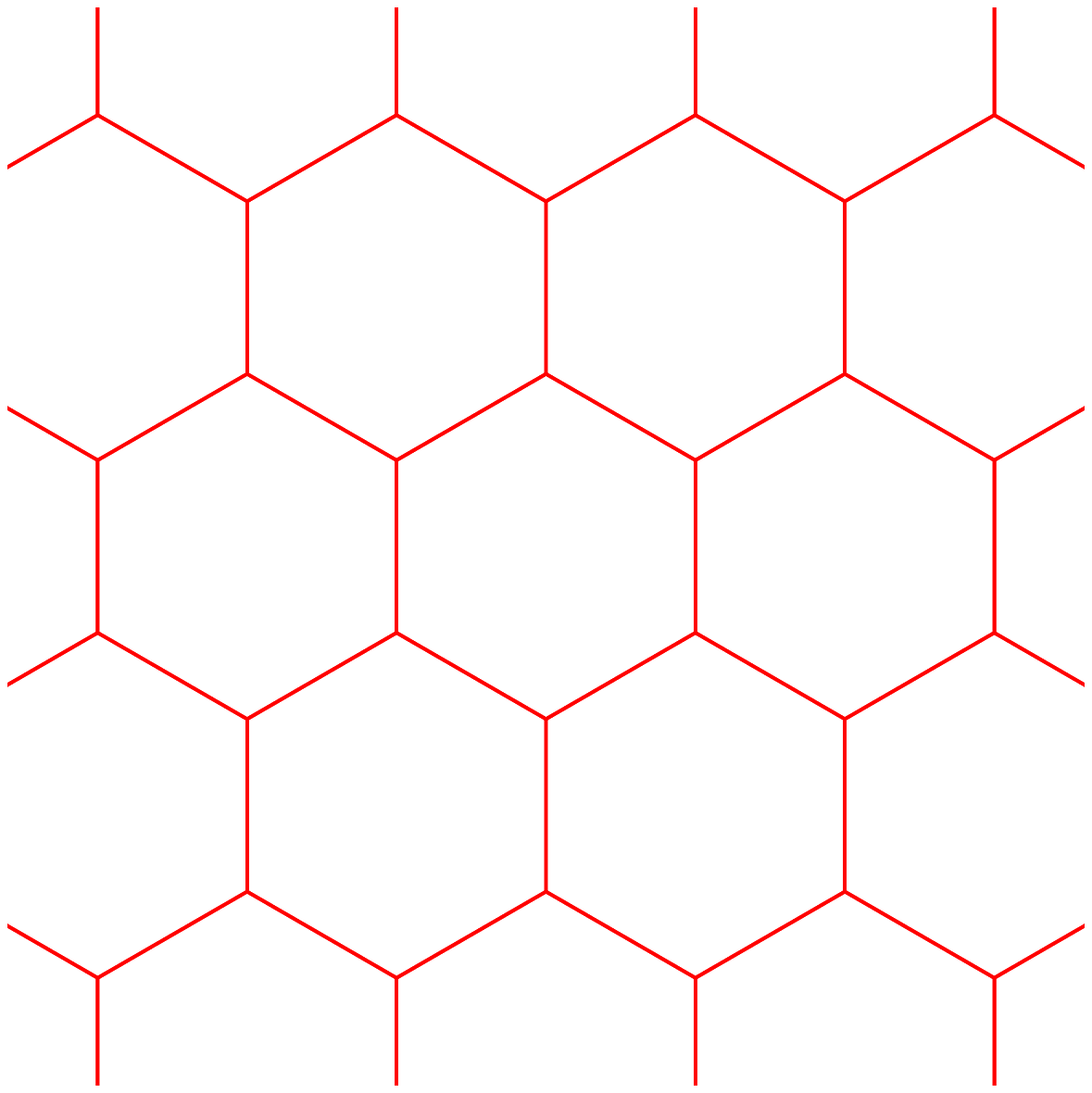}}\hskip 2pt\fbox{\includegraphics[scale=0.25]{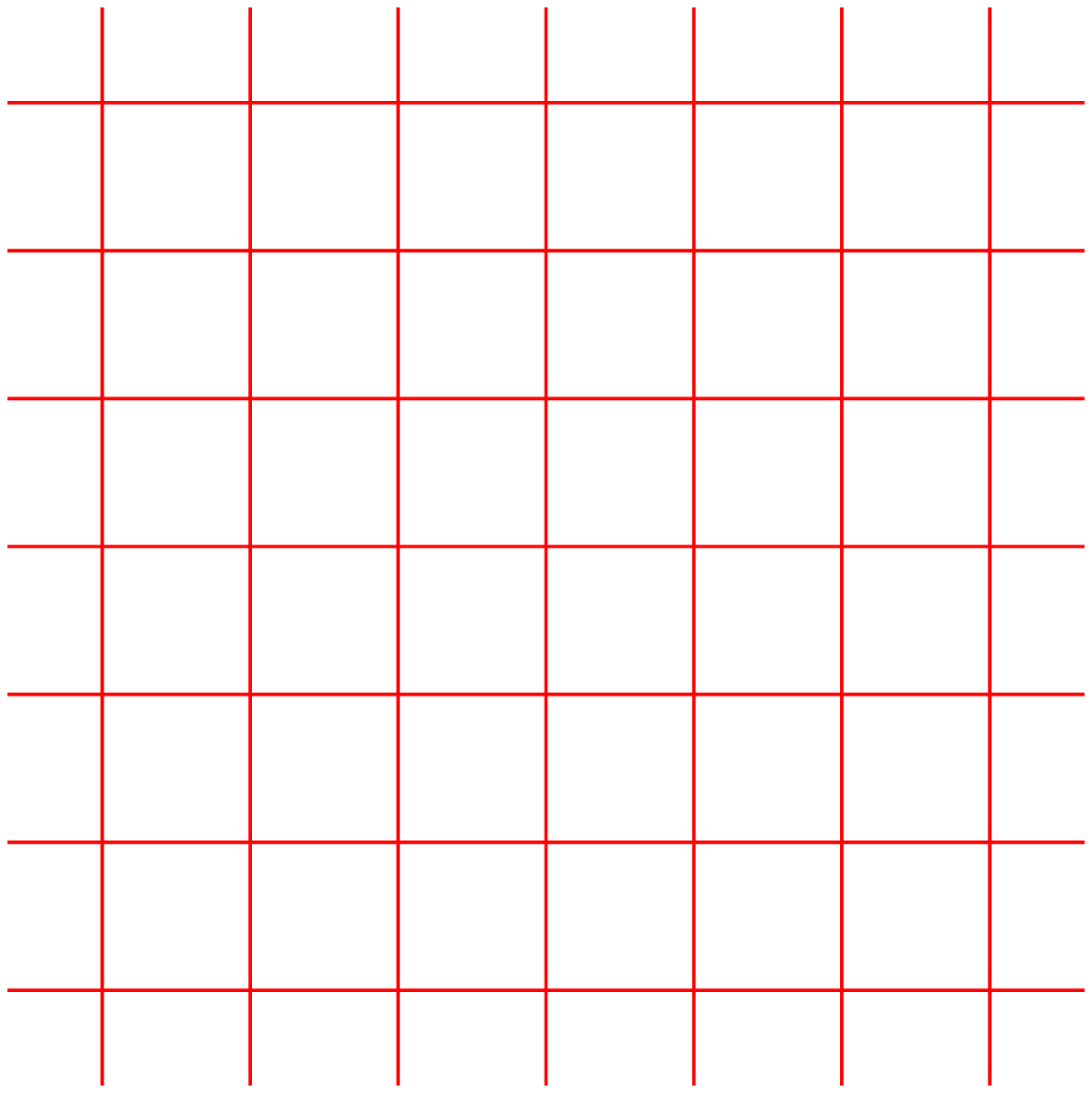}}\hskip 2pt\fbox{\includegraphics[scale=0.25]{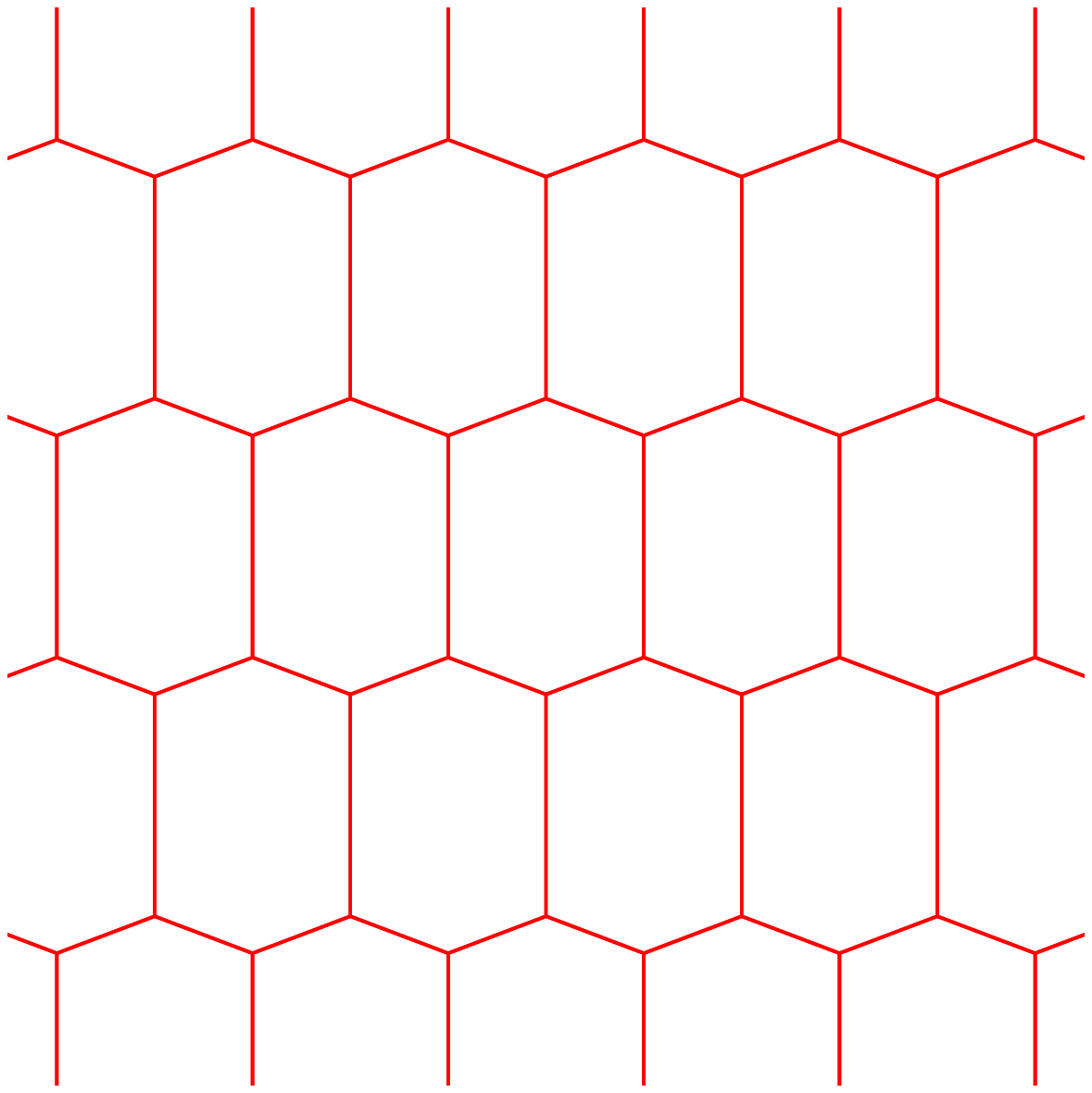}}\hskip 2pt\fbox{\includegraphics[scale=0.25]{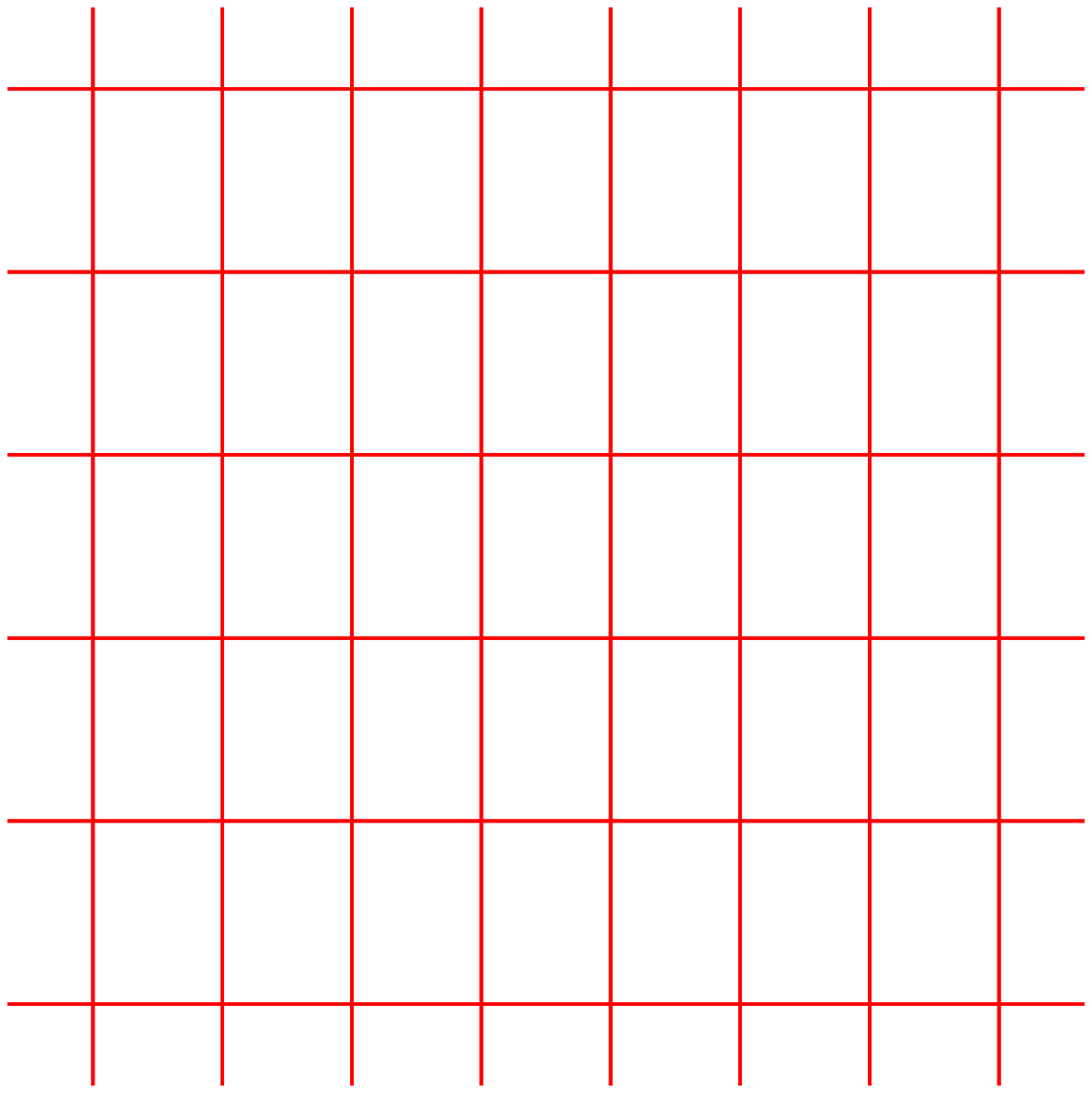}}\hskip 2pt\fbox{\includegraphics[scale=0.25]{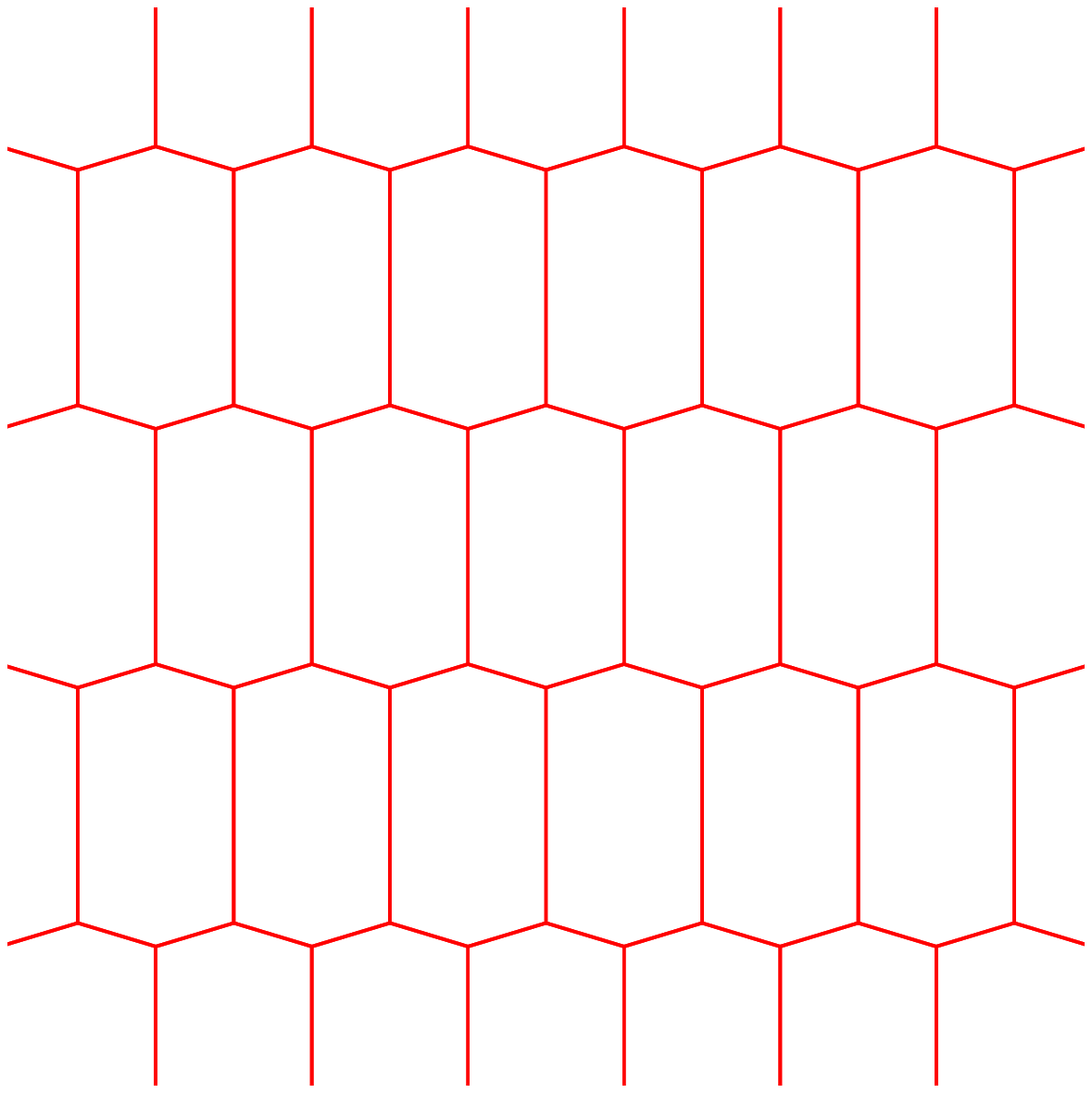}}
\end{center}
\caption{The tiled horospheres for the case of Euclidean domains,
  $-D=3,4,7,8,11$} 
\end{figure}

\begin{figure}[H]
\begin{center}
\fbox{\includegraphics[scale=0.4]{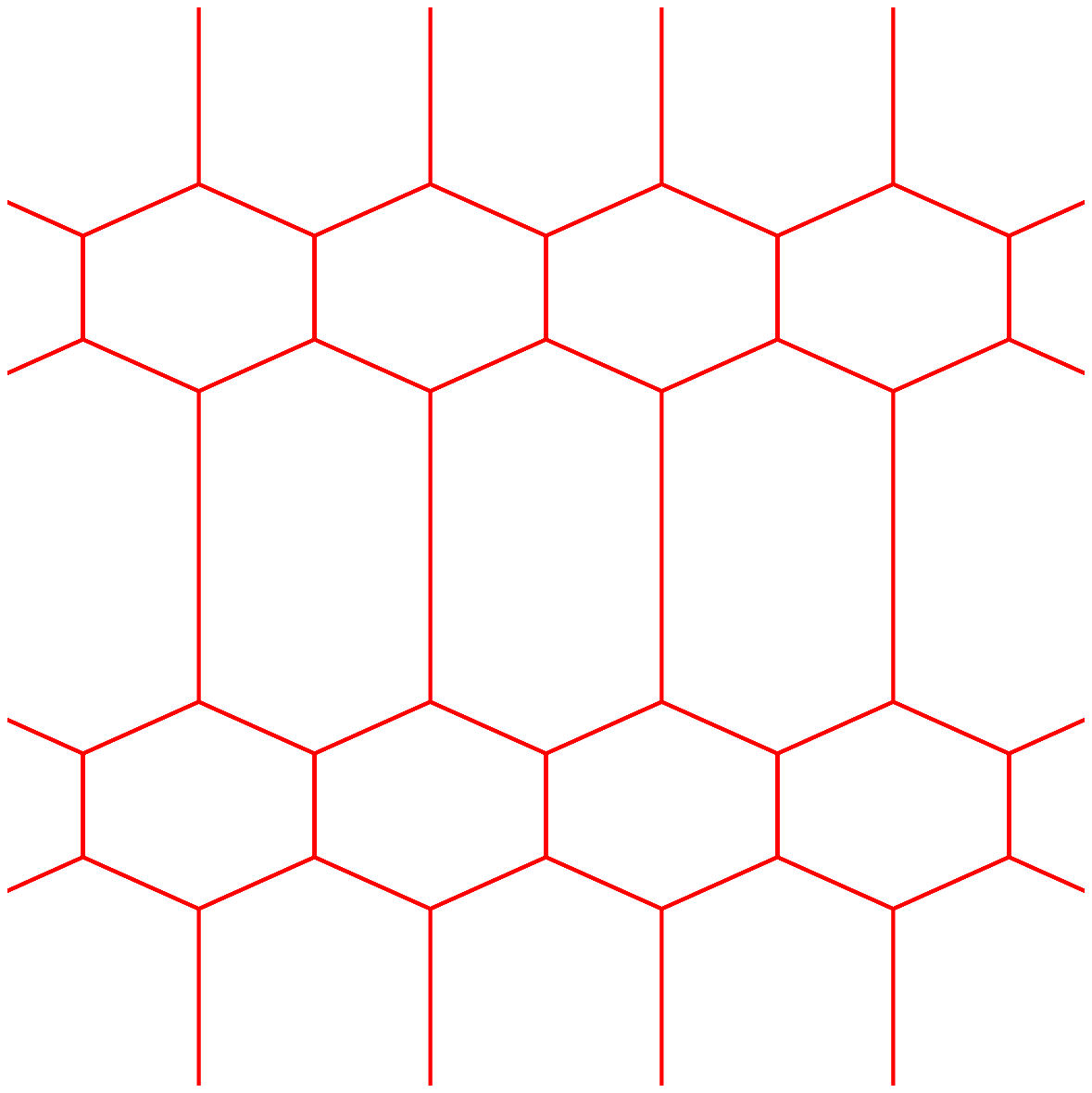}}\hskip 2pt\fbox{\includegraphics[scale=0.4]{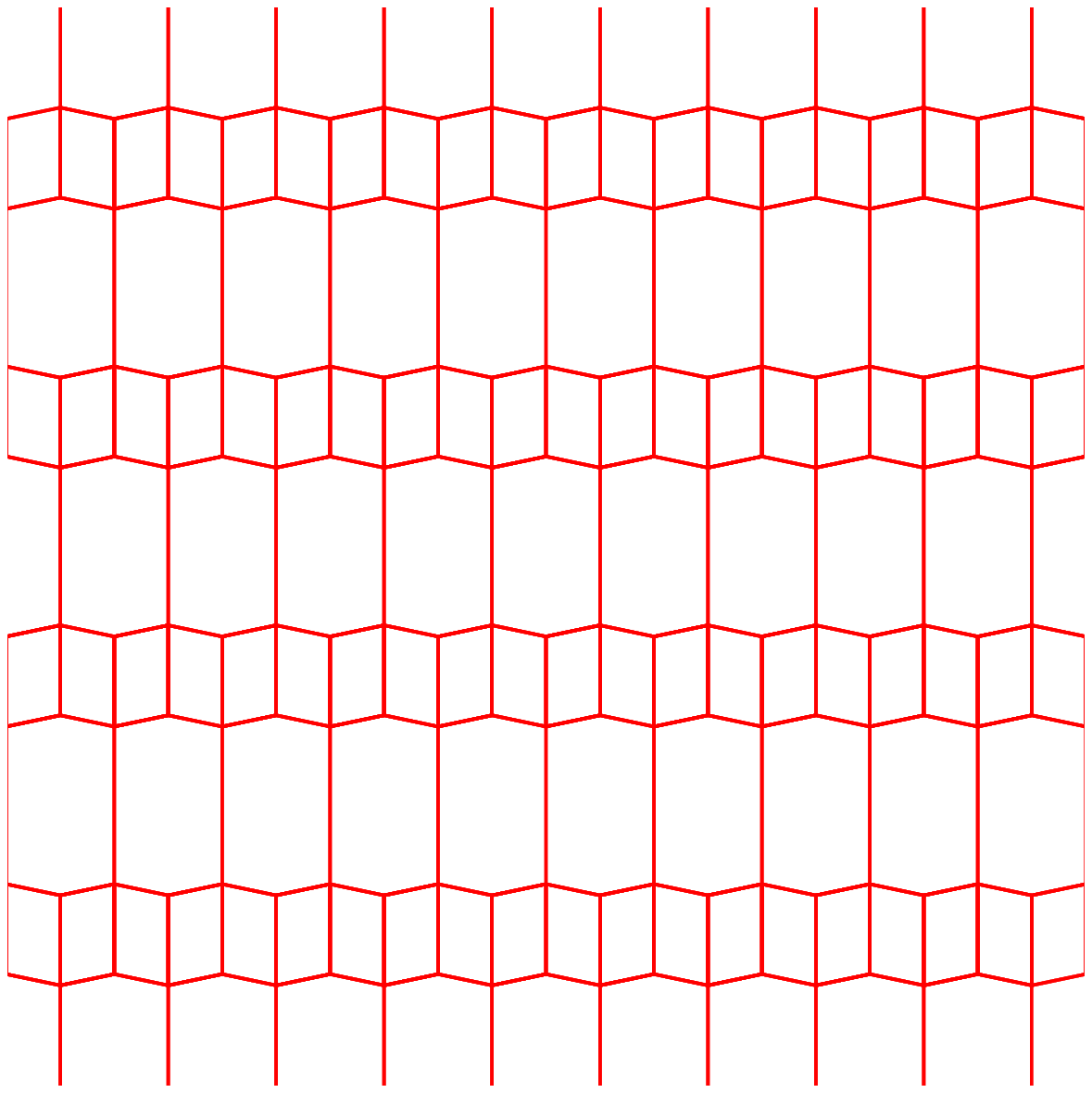}}\hskip 2pt\fbox{\includegraphics[scale=0.4]{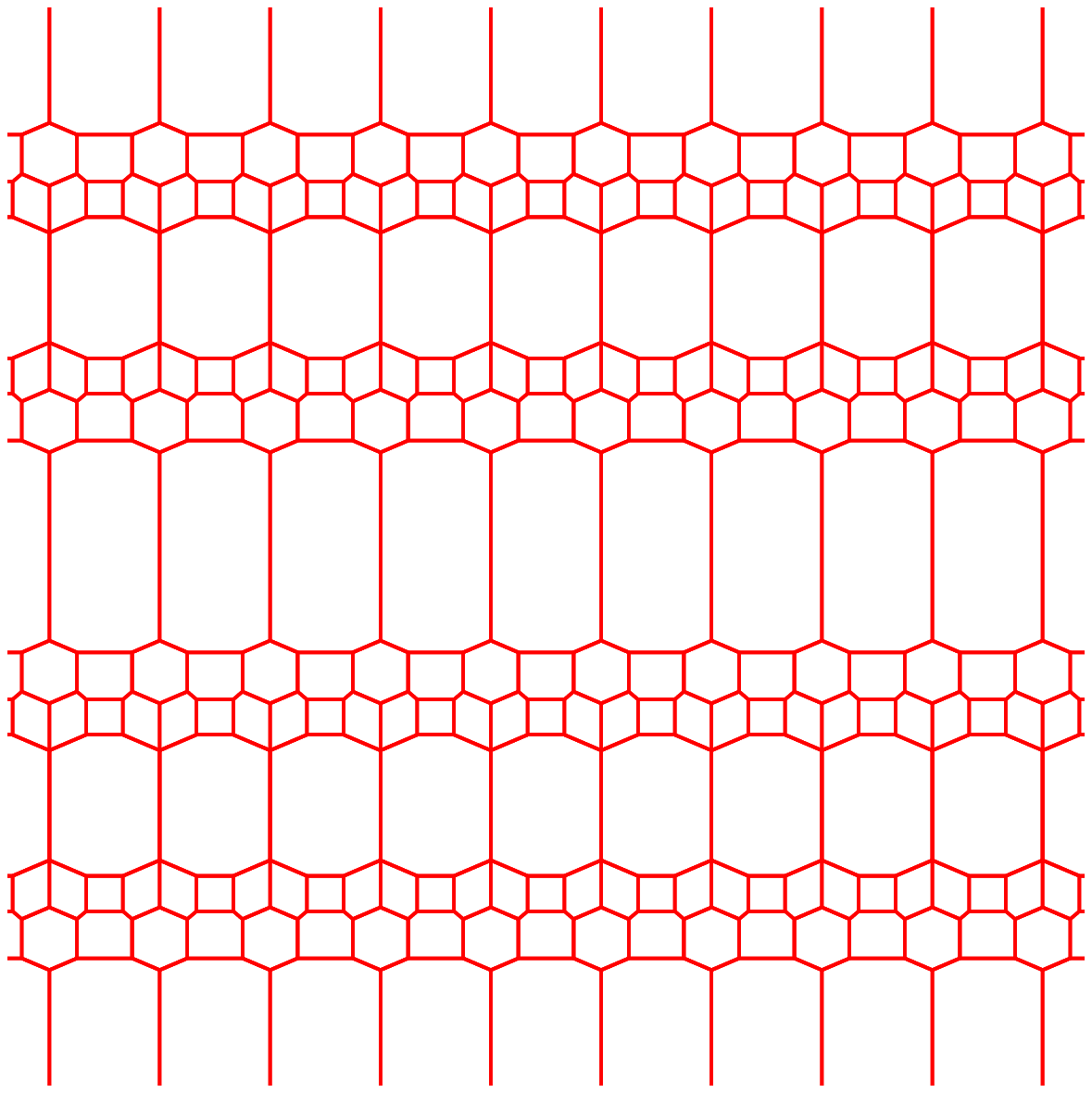}}
\end{center}
\caption{The tiled horospheres at the infinite cusp for the cases $-D=20,23,88$}
\end{figure}

\begin{prop} \label{firstkiss}
Let $\alpha=a/b$ and $\beta=c/d$. Let $I$ be the ideal generated 
by $a$ and $b$, and $J$ the ideal generated by $c$ and $d$. Then the horoballs 
$B_{\alpha}(1)$ and $B_{\beta}(1)$  are either disjoint or touch at one point. 
They touch if and only if $IJ$ is the principal ideal generated by $ad-bc$. 
\end{prop} 
\begin{proof} 
Let $r_\alpha= N(I)/2N(b)$ and $r_\beta=N(J)/2N(d)$ be Euclidean radii of 
the horoballs $B_{\alpha}(1)$ and $B_{\beta}(1)$ . As one easily checks, 
these two horoballs have a non-empty intersection if and only if 
\[
d(\alpha, \beta)^2 \leq 4 r_\alpha r_\beta, 
\] 
where $d(\alpha,\beta)$ is the Euclidean distance 
between $\alpha$ and $\beta$. The equality holds precisely when the two 
horoballs touch. Since $d(\alpha,\beta)=\frac{|ad-bc|}{|bd|}$, the 
inequality can be rewritten as
$N(ad-bc) \leq N(I)N(J) =N(IJ)$. On the other hand, $ad-bc$ is an element of 
the ideal $IJ$ and, therefore, $N(ad-bc) \geq N(IJ)$. 
The equality holds if and only if $IJ$ is generated by $ad-bc$. 
Proposition follows. 
\end{proof} 

Let $\partial H_{\alpha}$ denote the boundary of  $H_{\alpha}$ 
where $\alpha=a/b$. Let $I$ be the ideal generated by $a$ and
$b$. Let $J$ be the inverse of $I$. Then there exists $c$ and $d$ in
$J$ such that $ad-bc=1$. Note that this identity implies that $J$ is
generated by $c$ and $d$.  Let $\beta=c/d$.  Then by Proposition
\ref{firstkiss}, $\{\alpha,\beta \}$ parameterizes a cell in 
 $\partial H_{\alpha}$.  One easily checks that the linear
map defined by
\[
L
\left(\begin{array}{c} 
 a  \\
 b 
 \end{array}\right) =
\left(\begin{array}{c} 
 a \\
 b 
 \end{array}\right)  \text{ and } 
L\left(\begin{array}{c} 
 c  \\
 d
 \end{array}\right) 
 = -
\left(\begin{array}{c} 
 c  \\
 d 
 \end{array}\right) 
 \]
 is in $\GL_2(A)$. In particular, $L$ preserves the spine $X$.  
 Note that the Poincar\'e extension of $L$ to $\mathbb H^3$ is a 
$180^{\circ}$ rotation about the geodesic line connecting the two cusps. Thus 
$L$ induces a central symmetry of the cell. The center is the point
 where the two horoballs touch.

 In particular, there is a cell in 
$\partial H_\infty$ consisting of points equidistant to $\infty$ and $0$. We 
shall call this the {\em fundamental} cell and denote it by $\Sigma$. 

\begin{lemma} \label{reduction}
 Let $e$ be the edge of the fundamental cell $\Sigma$ which consists of 
points equidistant to $\infty$, $0$ and $\alpha=a/b$, where $a,b\in A$. Let $I$ be the ideal 
in $A$ generated by $a$ and $b$. Then 
\[ 
N(a), N(b) \leq  N(I)\frac{|D|}{2}
\]  
and the projection of the edge $e$ on the boundary 
$\bbC$ lies on the line $tr(z\bar{a}b) = N(a)+N(b)-N(I)$. 
\end{lemma} 
\begin{proof}
It is known (see \cite{Ha} and \cite{Me}) that 
\[
\mathbb H^3 = \cup_\beta B_{\beta}(\sqrt{\frac{|D|}{2}}). 
\] 
Since $\infty$, $0$ and $\alpha$ are the closest cusps to the edge $e$, 
\[
B_{\infty}(\sqrt{\frac{|D|}{2}}) \cap B_{0}(\sqrt{\frac{|D|}{2}}) \cap 
B_{\alpha}(\sqrt{\frac{|D|}{2}}) \neq \emptyset 
\] 
Since $B_{\infty}(\sqrt{\frac{|D|}{2}})$ is the half-space consisting of points 
$(z,\zeta)$ in $\mathbb H^3$ such that $\zeta\geq \sqrt{\frac{2}{|D|}}$ and 
$B_{\alpha}(\sqrt{\frac{|D|}{2}})$ is a Euclidean ball of radius 
$\sqrt{\frac{|D|}{2}}\frac{N(I)}{N(b)}$ touching the boundary $\bbC$, these two subsets 
have non-empty intersection if and only if $N(b) \leq N(I)\frac{|D|}{2}$. Similarly, the balls 
$B_{0}(\sqrt{\frac{|D|}{2}})$ and $B_{\alpha}(\sqrt{\frac{|D|}{2}})$ intersect if and only if 
$N(a) \leq N(I)\frac{|D|}{2}$.

The set of points equidistant to $\infty$ and $0$ is the Euclidean hemisphere 
centered at $0$ with radius $1$. Likewise, the set of points equidistant to $\infty$ and $\alpha$ is the Euclidean hemisphere centered at $\alpha$ with 
radius $\sqrt{\frac{N(I)}{N(b)}}$. An easy calculation now shows that 
the edge, which is contained in the intersection of these two hemispheres, 
projects to the line $tr(z\bar{a}b) = N(a)+N(b)-N(I)$.
\end{proof}

By a classical result every ideal class contains an integral ideal of 
norm less than or equal to $\sqrt{\frac{|D|}{3}}$. Thus 
the projection of the fundamental cell $\Sigma$ to the boundary 
$\bbC$ is obtained by drawing finitely many lines $tr(z\bar{a}b) = N(a)+N(b)-N(I)$
where $N(a)$ and $N(b)$ satisfy inequalities as in Lemma \ref{reduction}, and the norm of the ideal $I$ generated by $a$ and $b$ is bounded by 
$\sqrt{\frac{|D|}{3}}$. The connected component (of the complement to these lines) containing 0 is the projection of $\Sigma$. 

Recall that $A$ is a Euclidean domain if for any two elements $a,b\in A$ such that $b\neq 0$ 
there exist $c\in A$ such that $|\frac{a}{b}-c|<1$. 
This holds for $D=-3,-4, -7,-8$ and $-11$.

\begin{prop} $\GL_2(A)$ acts transitively on 2-cells in $X$ if and only if $A$ 
is a Euclidean domain. Moreover, in that case, the projection of the fundamental
 cell $\Sigma$ to $\bbC$ is the Voronoi cell of the lattice $A$, that is,  the set of points $z$ in
$\bbC$ such that $0$ is the closest lattice point.
\end{prop} 
\begin{proof}  Let $\sigma$ be a 2-cell in $\partial H_{\infty}$. 
Then, as one easily checks, $\sigma$ is in the $\GL_2(A)$-orbit of $\Sigma$ if and only if 
$\sigma=g(\Sigma)$ for 
\[ 
g=\left(\begin{array}{cc} 
1 & x \\
0 & 1 \end{array}\right)
\] 
where $x\in A$, i.e. $\sigma$ is an $A$-translate of the fundamental cell $\Sigma$. 

Assume first that $\GL_{2}(A)$ acts transitively on all 2-cells. 
Then any 2-cell in $\partial H_{\infty}$
is an $A$-translate of $\Sigma$. In the upper half space model of 
$\mathbb H^{3}$ the fundamental cell $\Sigma$ lies on the hemisphere centered at 
$0\in \mathbb C$ with radius 1. 
In particular, the projection of $\Sigma$ on $\mathbb C$ is contained in the unit 
circle. Thus unit discs centered at points in $A$ cover $\mathbb C$, i.e. $A$ is a Euclidean domain. For 
the converse we need the following: 

{\bf Claim:} Let $\alpha=\frac{a}{b}\notin A$ such that $|\frac{a}{b}-c|<1$
for some $c\in A$. Then there is no 2-cell in $X$ consisting 
of points equidistant to $\infty$ and $\alpha$. 

Without loss of generality we can 
assume that $c=0$. Then $0<|\alpha|<1$, by assumption. 
Let $\beta=\bar{b}/\bar{a}$.  Then  $\alpha$ and $\beta$ both give rise to the same 
line, as defined in Lemma \ref{reduction}. In other words,  the hemispheres 
centered at $0$, $\alpha$ and $\beta$ of radii $1$, $\sqrt{\frac{N(I)}{N(b)}}$ and 
$\sqrt{\frac{N(I)}{N(a)}}$, respectively, intersect in the same geodesic. Moreover, 
 the hemisphere centered at $\alpha$ is below 
the intersection of the other two. This shows that there is no 2-cell in $X$ 
parameterized by the pair $\{\alpha,\infty\}$, as claimed. 

Thus, if  $A$ is a Euclidean domain then 
open unit discs centered at points in $A$ cover $\bbC$, and 
the claim implies that every 2-cell in $\partial H_{\infty}$ is an 
$A$-translate of $\Sigma$. Moreover, since $\GL_2(A)$ acts transitively on cusps, 
it follows that $\GL_{2}(A)$ acts transitively on 2-cells in $X$.

Finally, if $A$ is a Euclidean domain, the projection of $\Sigma$ to $\bbC$ is cut out by the 
lines $tr(z\bar{a})=N(a)$ for all $a$ in $A$. But the line $tr(z\bar{a})=N(a)$ consists of 
points on $\bbC$ equidistant to $0$ and $a$. Thus $\Sigma$ is the Voronoi cell. 
\end{proof} 

 Assume that $A$ is a Euclidean domain. Then we have a bijection 
$\mathbb P(A)\cong\mathbb P(k)$. Thus  $\alpha$ in $\mathbb P(k)$ 
can be written as  $\alpha=\frac{a}{b}$ where $a$ and $b$ are relatively prime 
elements in $A$ and the cusp $\alpha$ naturally corresponds to a 
{\em lax vector} $(a,b)\in A^2$. If  $\alpha=\frac{a}{b}$ and 
 $\beta=\frac{c}{d}$ are two cusps such that $(a,b)$ and $(c,d)$ are 
lax vectors, then 
Proposition \ref{firstkiss}  
  says that $B_\alpha(1)$ and $B_\beta(1)$ 
are either disjoint or touch, and two touch if and only if 
two corresponding lax vectors span $A^{2}$, i.e. form a {\em lax basis}.  
 Therefore, lax bases parameterize  2-dimensional cells if $A$ is a Euclidean domain. 

If $D \equiv 0\pmod{4}$ then the Voronoi cell is a rectangle whose
sides bisect the segments $[0, \pm 1]$ and $[0,
  \pm\frac{\sqrt{D}}{2}]$.  If $D \equiv 1\pmod{4}$ then the Voronoi
cell is a hexagon whose sides bisect the segments $[0, \pm 1]$ and
$[0, \pm\frac{1}{2}\pm\frac{\sqrt{D}}{2}]$.  Thus, if $A$ is a
Euclidean domain,  then there is only one $\GL_2(A)$-orbit of
0-cells (vertices) in $X$, and one or two orbits of 1-cells, depending
on whether the edges of the Voronoi cell have equal lengths or not.

 We shall now describe the spine $X$ for $D=-3$ and $-4$. 
If $D=-3$ then $A=\bbZ[\rho]$ where $\rho=e^{\frac{2\pi i}{6}}$.  The fundamental 
cell $\Sigma$ is a  regular hexagon. Its vertices are $v_n$ ($n=1, \ldots ,6$) where 
$v_n$ is contained in exactly four regions $H_{\alpha}$,  where $\alpha= 0,\infty, \rho^{n-1}, \rho^n$. The 
link of every vertex of $X$ is the 1-skeleton of a tetrahedron. In particular, 
the spine $X$ consists of regular hexagons, where 3 are glued along each edge.
 If $D=-4$ then $A=\bbZ[i]$. The fundamental 
cell $\Sigma$ is a square. Its vertices are $v_n$ ($n=1, \ldots ,4$) where $v_n$ 
is contained in exactly six regions $H_{\alpha}$,  
where $\alpha= 0,\infty, i^{n-1}, i^n, i^{n-1}+i^n$ and $(i^{n-1}+i^n)/2$. The 
link of every vertex of $X$ is the 1-skeleton of a cube. In particular, 
the spine $X$ consists of squares, where 3 are glued along each edge.
These results are not new, as they can be traced back to an article of Speiser
\cite{Sp}. For more general values of $D$, see \cite{Me, Vo}. The
structure of $\mathbb H^3/GL_2(A)$ for $|D|<100$ was
determined by Hatcher \cite{hatcher}.

\section{Ocean} 

Let $f(x,y)$ be an integral hermitian form over $k$.  The form $f$
defines a function $F: \mathbb P(k) \rightarrow \mathbb Q$ (values on
the regions $H_{\alpha})$ as follows.  Let $\alpha=(a,b)\in \mathbb
P(I)$. Then
\[
F(\alpha)=\frac{f(a,b)}{N(I)} 
\]
is well defined, that is, does not depend on the choice of $I$ in the ideal class. 
Conversely, if we know $F(\alpha)$ and $\alpha$ is given as a fraction $\frac{a}{b}$ 
then we can compute the value $f(a,b)$. Note that
there exists an integer $n$ such that 
$F(\alpha)\in \frac{1}{n}\mathbb Z$ for every cusp $\alpha$, so the set of 
values of $F(\alpha)$ is discrete. 

\begin{deff} 
Let $f$ be an  indefinite and anisotropic integral  hermitian form over $k$. 
The ocean $C$ for the form $f$ is the subset of $X$ consisting 
of all 2-dimensional cells $\sigma_{\alpha,\beta}=H_{\alpha}\cap H_{\beta}$ such that $F(\alpha)$ and 
$F(\beta)$ have opposite signs. 
\end{deff} 

\smallskip 
We note that the ocean is always non-empty for any indefinite form $f$. 
Indeed, if not, then 
we can write $\mathbb H^{3}$
\[
\mathbb H^{3} = \bigcup_{F(\alpha)>0}H_{\alpha} \cup \bigcup_{F(\alpha)<0}H_{\alpha}
\]
as a disjoint union of two nonempty closed subsets, a contradiction. 

Let $U(f) \subseteq \GL_{2}(A)$ be the unitary group preserving the form $f$. 
Note that $U(f)$ acts on the ocean $C$ of the form $f$. 
The main objective of this section is the following. 

\begin{thm}\label{T1} Let $f$ be an indefinite and anisotropic integral hermitian form over $k$. Let 
$C$ be the corresponding ocean. Then the quotient $C/U(f)$ is compact. 
\end{thm}
\begin{proof} It suffices to show that the number of $U(f)$-orbits of vertices 
contained in $C$ is finite. We start with the following lemma. 

\begin{lemma} \label{vertex} Let $v$ be a vertex of $X$. Then the 
 form $f$ is determined by the values of $F$ on regions containing the 
 vertex $v$. 
\end{lemma}
\begin{proof} 
We can assume that the vertex belongs to a 2-cell $\sigma_{\alpha,\delta}$ 
and the two edges (meeting at the vertex) are also shared by 2-cells 
$\sigma_{\beta,\delta}$ and $\sigma_{\gamma,\delta}$ respectively. 
By transforming by an element of 
$\GL_{2}(k)$, if necessary, we can assume that $\delta=\infty$, that is, the three 
cells sit on hemispheres centered at $\alpha$, $\beta$ and
$\gamma$. Since the hemispheres intersect at one point (namely $v$) 
these three complex numbers are not on a line. 
Lemma follows from Proposition \ref{hermitian}. 
\end{proof}

Let $H_\alpha$ be region that intersects $C$. Then there is a region 
$H_{\beta}$ such that  $H_\alpha\cap H_\beta\cap C \neq \emptyset$ 
and $F(\alpha) F(\beta)< 0$. Write 
$\alpha=\frac{x}{y}$ and $\beta=\frac{z}{w}$ for some $x,y,z$ and $w$
in $A$. Let $I$ and $J$ be the ideals in $A$ generated by $x,y$ and
$z,w$ respectively. The positive rational number
\[
N_{\alpha,\beta}=\frac{N(xw-yz)}{N(I)N(J)}
 \]
 is well defined and invariant under the action of $\GL_{2}(A)$.
 (The expression $xw-yz$ is the determinant of the matrix 
 $\left(\begin{array}{cc} 
 x & z  \\
 y & w \end{array}\right)$. Geometrically, this number represents the
 exponential of the distance between the horospheres associated to the
 cusps $\alpha$ and $\beta$.)
 Since there are finitely many $\GL_2(A)$-orbits of 
 intersecting pairs $(H_\alpha,H_\beta)$, 
 $N_{\alpha,\beta}$ takes a finite number of values.  Let 
 \[
\bfu= \left(\begin{array}{c} 
 x \\
 y \end{array}\right) 
 \text{ and }
\bfv= \left(\begin{array}{c} 
 z \\
 w \end{array}\right) \]
and assume that in the (vector space) basis $(\bfu,\bfv)$ the matrix of the 
form $f$ is $\left(\begin{smallmatrix}a&\bar\nu\\
\nu&c\end{smallmatrix}\right)$. 
Also assume that in the standard basis of $A^2$ the matrix of the form $f$  is 
$\left(\begin{smallmatrix}a'&\bar\nu'\\
\nu'&c'\end{smallmatrix}\right)$. Then 
\[ 
\left(\begin{array}{cc} 
x & z \\
y & w 
\end{array}\right) 
\left(\begin{array}{cc}
a' & \nu' \\
\bar{\nu}' & c' 
\end{array}\right)
\left(\begin{array}{cc}
\bar{x} & \bar{y} \\
\bar{z} & \bar{w}  
\end{array}\right)
=
\left(\begin{array}{cc}
a & \nu \\
\bar{\nu} & c
\end{array}\right). 
\] 
After taking the determinant of both sides and multiplying by $D$ we obtain  
\[
N(xw-yz) \Delta= D(ac- N(\nu))
\]
where $\Delta$ is the discriminant of the form $f$. 
Since $-D N(\nu)\geq 0$, it follows that $N(xw-yz)\Delta\geq Dac$. 
After dividing by $N(I)N(J)$ we arrive at $N_{\alpha,\beta}\Delta\geq DF(\alpha)F(\beta)$. 
Moreover, since $F(\alpha)F(\beta)<0$, we also have $DF(\alpha)F(\beta)>0$. 
These two inequalities imply that $F$ takes only 
finitely many values on regions $H_\alpha$ that intersect $C$. 

We can now finish easily. Let $v$ and $u$ be two vertices in $C$ such
that $u=g(v)$ for some $g$ in $\GL_{2}(A)$. If
$F(\alpha)=F(g(\alpha))$ for all cusps $\alpha$ such that the region
$H_{\alpha}$ contains $v$, then $g$ is in $U(f)$ by Lemma
\ref{vertex}. Since there are finitely many $\GL_{2}(A)$-orbits of
vertices in $X$ and $F$ takes only finitely many values along the
ocean $C$, the group $U(f)$ has a finite number of orbits of vertices
in $C$. The theorem is proved.
\end{proof} 

Note that we have in the process obtained the following result due to Bianchi: 

\begin{cor} The number of $\GL_{2}(A)$-equivalence classes of
integral, indefinite hermitian
forms with fixed discriminant is finite. 
\end{cor}

\section{The Main Theorem}\label{s:main}

\def\C{{\mathbb C}}
\def\H{{\mathbb H}}
\def\R{{\mathbb R}}
\def\Z{{\mathbb Z}}

Let $C$ be the ocean (for a fixed $D$ and a fixed indefinite and
anisotropic integral hermitian form $f$). By $\H_f$ denote the totally
geodesic hyperbolic plane in
$\H^3$ whose boundary at infinity is the circle $\{f=0\}$. Finally,
let $\pi:C\to \H_f$ be the nearest point projection.

\begin{thm}
$\pi:C\to \H_f$ is a homeomorphism.
\end{thm}

The proof occupies the rest of the section. It is divided into
8 steps. Recall that by definition $C$ the union of 2-cells in the spine
$X$ such that $f$ assigns to the associated two cusps values with opposite
signs. 

{\it Claim 1.} $\pi$ restricted to each 2-cell is an embedding onto a
convex hyperbolic polygon.

In the proof we will use the fact that when $P,Q$ are two totally
geodesic hyperbolic planes in $\H^3$ that do not intersect
perpendicularly, then the nearest point projection $P\to Q$ is
injective (and in fact an embedding to an open subset).

Given a point $z$ in the interior of a 2-cell $\sigma$ in $C$, place
$z$ at the center of the ball model for $\H^3$.  Arrange that the
circle $f=0$ on the 2-sphere $S^2_\infty$ is a ``parallel'', i.e. a
circle equidistant from the equator in the spherical metric. Also
arrange that $\{f>0\}$ includes the north pole and $\{f<0\}$ includes the
south pole. Thus $\pi$ will map $z$ vertically to $\H_f$, which is the
convex hull of $\{f=0\}$.  There are two cusps equidistant from $\sigma$,
one positive (and hence north of $\{f=0\}$), the other negative (and hence
south of $\{f=0\}$). In particular, the two cusps do not belong to the
same parallel and thus the great circle $R_z$ bisecting the two cusps
is not perpendicular to the equator nor to $f=0$. Since the cell
$\sigma$ lies in the convex hull of $R_z$, by the above fact
$\pi|\sigma$ is injective. Finally, $\sigma$ is convex and the edges
of $\sigma$ are geodesic segments, so the same is true for the images
under $\pi$. \qed

{\it Claim 2.} Let $z$ be a point in the interior of a 2-cell
$\sigma\subset C$ and $\ell$ the geodesic through $z$ perpendicular to
$\H_f$. Thus $\ell$ has one (positive) end in $\{f>0\}$ and the other
(negative) end in  $\{f<0\}$.
Then a small open subinterval of $\ell$ with one endpoint at $z$ and
the other endpoint in the direction of the positive [negative] end is
contained in the interior of some $H_\alpha$ with $f(\alpha)>0$
[$f(\alpha)<0$]. 

Indeed, $z$ has precisely two closest cusps, one positive one
negative. In the setting of the proof of Claim 1, the ray going north [south]
from $z$ is contained on the same side of the bisecting plane
$Hull(R_z)$ as the positive [negative] cusp. The Claim follows from
the fact that for every point in a small neighborhood of $z$ the
closest cusp is one of these two.\qed

{\it Claim 3.} Every edge $e$ in $C$ belongs to an even number of 2-cells
in $C$.

This is because the cells of $C$ that contain $e$ are arranged in a
cyclic order around $e$ with sectors determined by consecutive cells
belonging to some $H_\alpha$. A 2-cell is in $C$ precisely when the
signs at the two adjacent sectors are opposite. In a cyclic arrangement of
signs there is always an even number of sign changes.\qed

{\it Claim 4.} If $\sigma,\tau$ are two distinct 2-cells in $C$, then
$\pi(\sigma)$ and $\pi(\tau)$ have disjoint interiors.

For suppose not. Choose a point $y\in int~ \pi(\sigma)\cap int~
\pi(\tau)$ that is not in the image of the 1-skeleton of $C$. Let
$\ell$ be the geodesic through $y$ perpendicular to $\H_f$. Thus $\ell$
intersects $\sigma,\tau$ in two interior points $z,w$ and all
intersection points of $\ell$ and $C$ occur in the interiors of
2-cells and form a locally finite set (in fact, finite -- see Claim
7). After renaming, we will assume that $z,w$ are two consecutive
intersection points in $\ell$. By Claim 2, small neighborhoods of $z$
and $w$ in the segment $[z,w]\subset\ell$ belong to regions $H_\alpha$
and $H_\beta$ with opposite signs. But this means that $[z,w]$
intersects $C$ in an interior point, contradiction.\qed

{\it Claim 5.} Every edge $e$ in $C$ is contained in exactly two
2-cells in $C$ and their union is embedded by $\pi$.

Indeed, it is impossible for three polygons with a common side to have
pairwise disjoint interiors, so by Claim 3 the edge $e$ belongs to two
2-cells. Their images are two convex hyperbolic polygons with a common
edge and disjoint interiors, so $\pi$ embeds the union.\qed

{\it Claim 6.} $C$ is a surface and $\pi$ is a local homeomorphism.

First note that Claims 1 and 5 imply that $C$ is a surface and $\pi$
is a local homeomorphism away from the vertices. Let $v$ be a vertex
of $C$. By Claim 5, the link $Lk(v,C)$ is a disjoint union of
circles. Each component of $Lk(v,C)$ corresponds to a collection of
2-cells cyclically arranged around the common vertex $v$. By Claim 4
their images are cyclically arranged around $\pi(v)$ and $\pi$ is an
embedding on the union of these 2-cells. If there is more than one
circle in $Lk(v,C)$ there would be 2-cells contradicting Claim
4. (More formally, $Lk(v,C)$ can be identified with the space of germs
of geodesic segments in $C$ leaving $v$ and $\pi$ induces a map
$Lk(v,C)\to Lk(\pi(v),\H_f)=S^1$ with $Lk(\pi(v),\H_f)$ defined
similarly. Claim 5 implies that this map is locally injective, hence a
local homeomorphism, hence a covering map. By Claim 4 this map is
generically 1-1, hence a homeomorphism.)  Thus $Lk(v,C)$ is a single
circle, so $C$ is a manifold, and $\pi$ is a local homeomorphism.\qed

{\it Claim 7.} $\pi:C\to \H_f$ is a proper map, i.e. preimages of compact
sets are compact.

This follows from the fact that $\pi$ is $U(f)$-equivariant, and that
$U(f)$ acts cocompactly on $C$ (Theorem \ref{T1}), and discretely on
$\H_f$.\qed

{\it Claim 8.} $\pi:C\to \H_f$ is a homeomorphism.

A proper local homeomorphism between locally compact spaces is a
covering map. By Claim 4 this map is generically 1-1.\qed

\begin{cor} \label{negative curvature}
The ocean always contains vertices with negative curvature (i.e. the
sum of the angles around the vertex is $>2\pi$).
\end{cor}

\section{Hermitian forms, case $D=-4$}\label{gauss}

In the spirit of Conway, in this section we study the spine $X$ 
  over the Gaussian integers. We give an example of the ocean, and describe the corresponding group 
 $U(f)$ in terms of generators and relations.

Recall the description of the spine 
$X$ from Section \ref{Spine}. Its 2-cells are squares and correspond to lax bases of $A^2$. 
 The link $Lk_X(v)$ of any vertex $v$ is the 1-skeleton of a cube. 
 It follows that 2-cells are glued three along each edge.
 In particular, 1-cells in $X$ correspond to
 lax superbases, i.e. triples of lax vectors such that any two form a lax basis.
  If a 2-cell (a square) corresponds to 
 a lax basis $\{\bfu,\bfv\}$, then its edges correspond (consecutively) to superbases 
 $\{\bfu,\bfv, \bfu+i^k \bfv\}$ for $k=1, \ldots ,4$.

  The parameterization of vertices in terms of lax vectors is more complicated.
  Note that six regions, corresponding to the six faces of the cube representing
  $Lk_X(v)$, meet at the vertex $v$.  Moreover, any two regions
  corresponding to opposite sides meet at the vertex $v$ only. Thus
  the vertex $v$ corresponds to six lax vectors, and 
  any two lax vectors $\bfr$ and $\bfs$, corresponding to a pair of
  opposite sides of the cube, span a submodule of index
  $2$. Furthermore, the remaining four lax vectors can be expressed
  in terms of $\bfr$ and $\bfs$ as follows:
\[
\frac{1+i}{2}(\bfr+\bfs) , \frac{1+i}{2}(\bfr+i\bfs), \frac{1+i} {2}(\bfr-\bfs)\text{ and } \frac{1+i}{2}(\bfr-i\bfs). 
\]
The map $\bfr \mapsto \bfr$ and $\bfs\mapsto i\bfs$ induces a
$90^{\circ}$ rotation of the cube about an axis through the centers of
opposite faces. These rotations, for any pair of opposite sides of the
link square, generate the group $S_4$ of the orientation preserving
isometries of the cube, which is the
stabilizer of the vertex $v$ in $\PGL_2(A)$. The fundamental domain
for the action of $\PGL_2(A)$ on the spine $X$ is a $45^{\circ} -
45^{\circ} - 90^{\circ}$ triangle obtained by the barycentric subdivision
of the square. Now one easily checks the following:

\begin{prop} \label{cube}
Let $f$ be a hermitian form and $v$ a vertex of the spine $X$ 
determined by a pair of lax vectors $\{\bfr,\bfs\}$ such that $\bfr$ and $\bfs$ 
span an $A$-submodule of $A^2$ of index $2$. Then 
\[
2f(\bfr)+2f(\bfs)= \sum_{k=1}^{4} f(\bfr+i^k \bfs)
\]
i.e. twice the sum of the values $f$ at two opposite regions is equal to the sum of 
values of $f$ at the four remaining regions. In particular, $inv(v)=f(\bfr) +f(\bfs)$ 
is independent of the choice $\{\bfr,\bfs\}$ of opposite regions at the vertex $v$. 
\end{prop} 

 The proposition shows that it is not possible that 
 $f$ is positive at $\bfr$ and $\bfs$ while negative at the other four 
regions meeting at the vertex $v$. This shows that the ocean is a 
surface if $D=-4$. 

\smallskip 

Let $e$ be an edge of the spine $X$ corresponding to a superbasis $\{\bfu, \bfv, \bfv+\bfu\}$. 
Let $a$, $b$ and $c$ be the values of $f$ at $\bfu$, $\bfv$ and $\bfu+\bfv$, respectively. 
Let $v$ be a vertex of $e$, and let $z=inv(v)$. Then the six values of $f$ at the vertex $v$ 
are $a,b,c, z-a,z-b$ and $z-c$. In particular, $f$ is determined by the four values $a,b,c$ and $z$
 and, as one easily checks, the discriminant of $f$ is given 
by the formula 
\[ 
\Delta= 2a(a-z)+2b(b-z)+2c(c-z) + z^2. 
\] 
Next, note that the vertices of $e$ correspond to pairs (of opposite regions) 
$\{\bfu+\bfv , \bfu+i\bfv\}$ and $\{\bfu+\bfv , \bfu-i\bfv\}$. 
The following proposition relates  the values of $f$ at the two vertices of $e$. 

\begin{prop} \label{edges} 
Let $f$ be a hermitian form and $e$ an edge of the spine $X$ 
corresponding to a superbasis $\{\bfu, \bfv, \bfv+\bfu\}$.  Let $v$ and $v'$ be the 
vertices of $e$ that correspond to the pairs of opposite regions 
$\{\bfu+\bfv , \bfu+i\bfv\}$ and $\{\bfu+\bfv , \bfu-i\bfv\}$, respectively. 
Then 
\[
f(\bfu+i\bfv)+f(\bfu-i\bfv)=2[f(\bfu)+f(\bfv)]
\]
and 
\[ 
inv(v) + inv(v')= 2[f(\bfu) + f(\bfv) + f(\bfu +\bfv)]. 
\] 
\end{prop} 
 \begin{proof}  
 The first relation follows from the parallelogram law. The second is obtained from the first by 
 adding $2f(\bfu+\bfv)$  to both sides. 
 \end{proof}

Any edge $e$ is determined by a lax superbasis, and the sum of values
of $f$ at the three vectors is denoted by $inv(e)$. If $e$ connects
$v$ and $v'$ then, by Proposition \ref{edges},
\[ 
inv(v)+inv(v')=2 \cdot inv(e).
\]

\subsection{Example: $\Delta=6$}
  We shall now compute the ocean $C$, give a presentation 
of the group $PU(f)$ and work out the tiling of $\mathbb H^2$ obtained by projecting cells 
of $C$  for the hermitian form $f$ given 
by the matrix 
\[
\left(\begin{array}{cc} 
1 & \frac{1-i}{2} \\
\frac{1+i}{2} & -1 
\end{array}\right). 
\] 
The discriminant of $f$ is $\Delta =6$. Since the quaternion algebra  
$\left(\frac{-4, 6}{\bbQ}\right)$ is ramified
at primes $2$ and $3$, the form $f$ does not  
represent $0$.

 Let $(a,c)$ be a pair of values $a<0<c$ of $f$ at two vectors of a lax basis corresponding to a 
cell $\sigma$ in the ocean. Since 
\[
\Delta > Dac >0 
\]
 $(a,c)=(-1,1)$ is the only possibility. There are three types of vertices in the ocean:
$inv(v)=2$, where $f$ takes values (we pair values corresponding to
opposite sides of the cube) $(-1,3)$ $(1,1)$ and $(1,1)$; $inv(v)=0$,
$(-1,1)$, $(-1,1)$ and $(-1,1)$; $inv(v)=-2$, $(-3,1)$, $(-1,-1)$,
$(-1,-1)$. Vertices with $inv(v)=0$ have negative curvature. 
The stabilizer of a vertex $v$ in $PU(f)$ consists of orientation preserving symmetries 
 of the cube such that $f(g(\bfv))=f(\bfv)$ for all lax vectors $\bfv$ at $v$. In particular, 
it is a cyclic group of order $3$ if $inv(v)=0$  and it is a
cyclic group of order $4$ for the other two types of vertices. 
There are two types of edges in the ocean: $inv(e)=\pm 1$, and one type of 2-dimensional
cells, pictured below, with $inv(v)$ given for every vertex.

\begin{picture}(300,180)(-140,-20)

\thicklines

\put(40,40){\line(1,0){60}}
\put(40,40){\line(0,1){60}}
\put(40,100){\line(1,0){60}}
\put(100,40){\line(0,1){60}}

\put(30,30){$0$}
\put (30,100){$2$}
\put(100,30){$-2$}
\put(105,100){$0$}

\end{picture}

Under the group $U(f)$, we have 3
orbits of vertices, 2 orbits of edges, and one orbit of 2-dimensional
cells. This can be seen as follows. For the vertices, assume that 
  $inv(v)=inv(v')$. Take $g\in \GL_{2}(A)$ such that $v'=g(v)$. Since the stabilizer of $v'$ in
$\PGL_{2}(A)$ is the group of orientation preserving symmetries of a
cube we can adjust $g$, if necessary, so that
$f(g(\bfv))=f(\bfv)$ for all lax vectors $\bfv$ at $v$, i.e. $g\in
U(f)$.  (Note that for any $v$ in $C$ there is an
orientation reversing symmetry of the cube that preserves the values
on the faces.) Any 2-dimensional cell in $C$ has a vertex 
$v$ such that $inv(v)=2$. Since $U(f)$ acts transitively on vertices with $inv(v)=2$, we can 
move the cell so that $v$ is any (fixed) vertex with the invariant 2. 
There are 4 cells in $C$ containing $v$ and the stabilizer of $v$ 
in $PU(f)$ acts transitively on them. This shows that $PU(f)$ acts transitively on 2-dimensional 
cells and on each of the two types of edges.

It follows that any 2-cell is is a fundamental domain for $PU(f)$. 
The quotient $C/PU(f)$ is a sphere (more precisely, an orbifold) obtained by identifying the 
edges of the 2-cell with the same invariant. 
The covering $\pi : C \rightarrow C/PU(f)$ has ramification indices $3$,
$4$ and $4$ at the vertices with $inv(v)=0$, $inv(v)=-2$ and $inv(v)=-2$, respectively. Thus 
$PU(f)$ has the following presentation:
\[
\langle r,s,t ~|~ t^3=r^4=s^4=rst=1\rangle. 
\] 

\begin{figure}
\begin{center}
\includegraphics[scale=0.8]{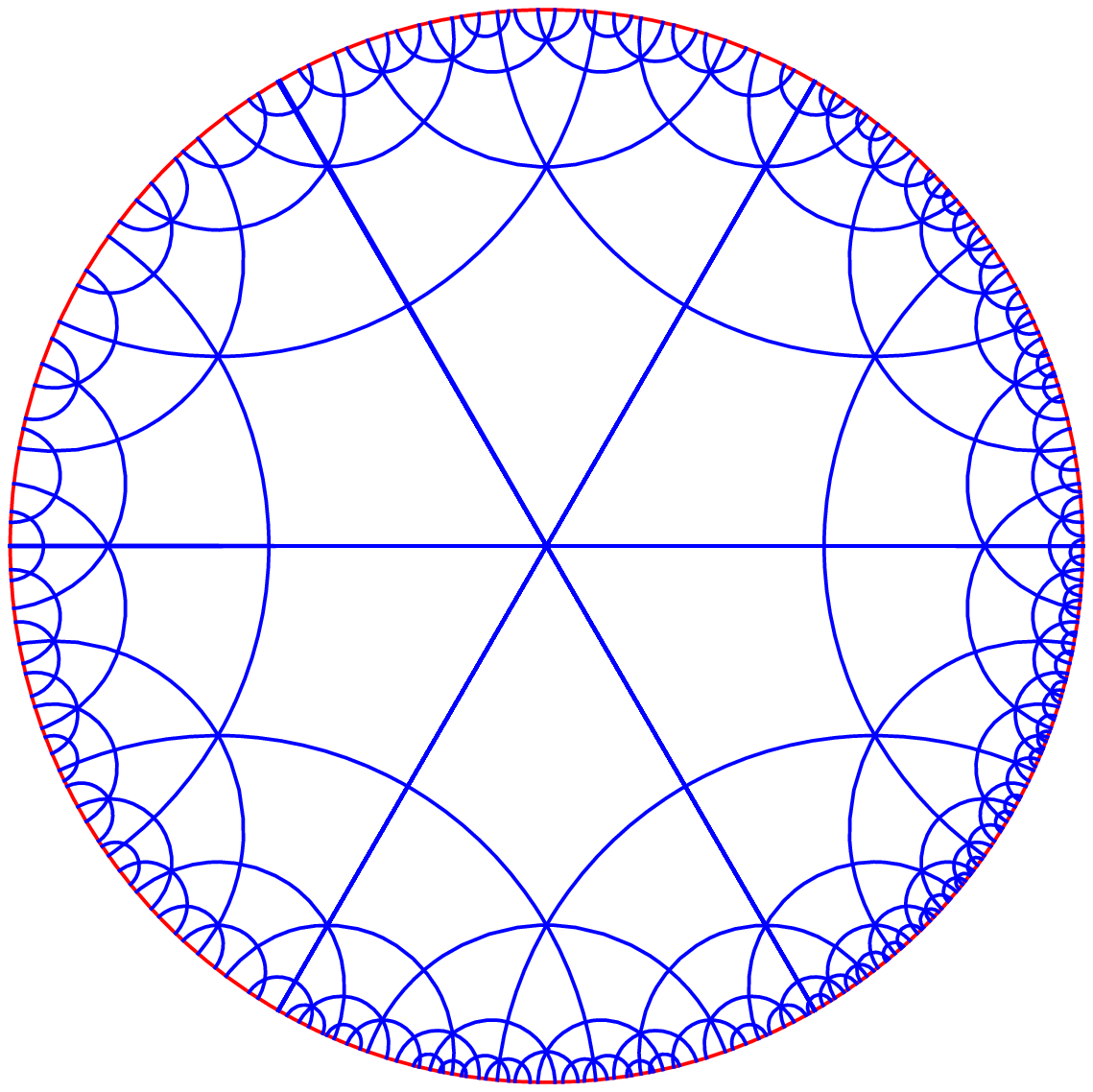}
\end{center}
\caption{\label{escher} The image of the ocean under the nearest point
  projection to the convex hull of $f=0$, for $D=-4$, $\Delta=6$.}
\end{figure}

We now work out the tiling of $\H^2$ obtained by projecting the cells
of $C$.
Let $v$ be a vertex in $C$ such that $inv(v)=0$. Since the stabilizer
of $v$ in $\PGL_{2}(A)$ is equal to the group of orientation preserving
symmetries of a cube, there is $g\in \GL_{2}(A)$ that exchanges
positive faces with negative faces.  Thus $f(g(\bfv))=-f(\bfv)$ for
every lax vector at $v$, so $f\circ g=-f$. (The subgroup of
$Stab_{PGL_2(A)}(v)$ that sends $f$ to $\pm f$ is the dihedral group
of order 6. To see this, note that the edges of the cube that belong to
both positive and negative faces form a hexagon.)

Let $\pi : C \rightarrow \mathbb H^{2}_{f}$ be the nearest point
projection.  Since $\mathbb H^{2}_{f}=\mathbb H^{2}_{-f}$, and $f\circ
g=-f$, for every edge $e$ in $C$, $g\circ \pi (e)=\pi(e')$ where $e'$
is an edge of the different invariant. It follows that $\pi(\sigma)$
is equilateral, for every 2-dimensional cell $\sigma\in C$. Since 6 of
these meet at $\pi(v)$ if $inv(v)=0$, and 4 meet at $\pi(v)$ if
$inv(v)=\pm 2$, $\pi(C)$ gives a tiling of $\mathbb H^{2}_{f}$ with
$90^{\circ}-60^{\circ}-90^{\circ}-60^{\circ}$ rhombi. It is the tiling
underlining Escher's Circle Limit I. See the figure.

\section{Hermitian forms, case $D=-3$}\label{eisenstein}

In the spirit of Conway, in this section we study the spine $X$ 
  over the Eisenstein integers.  As an application we prove 
 two sharp results on the minima of hermitian forms.

 \smallskip 
 
 Let $A=\bbZ[\rho]$ where $\rho=e^{\frac{2\pi i}{6}}$.  
 Recall that the spine $X$ has  
 a particularly nice description. Its 2-cells are regular hexagons. 
 The link $Lk_X(v)$ of any vertex $v$ is the 1-skeleton of  
 a tetrahedron. It follows that 2-cells are glued three along each edge.  
 In particular, 1-cells in $X$ correspond to 
 lax superbases, i.e. triples of lax vectors such that any two form a lax basis.  
If a 2-cell (a hexagon) corresponds to 
 a lax basis $\{\bfu,\bfv\}$, then its edges correspond (consecutively) to superbases 
 $\{\bfu,\bfv, \bfu+\rho^i \bfv\}$ for $i=1, \ldots 6$.

 Vertices of $X$ correspond to lax ultrabases, i.e. quadruples of lax vectors such that any three  form a lax superbase.  A typical ultrabasis is 
 \[
 \{\bfu,\bfv,\bfu+\bfv, \bfu +\rho \bfv\}. 
 \]
 The stabilizer in $\PGL_2(A)$ of any vertex of $X$ is the tetrahedral group $A_4$. 
 Indeed, the following matrix in the basis $(\bfu,\bfv)$ 
 \[
 \left(\begin{array}{cr} 
 1 & -1 \\
 1 &   0 
 \end{array}\right)
 \]
 cyclicly permutes the lax vectors $\bfu$, $\bfv$ and $\bfu+\bfv$ while
 keeping fixed the lax vector $\bfu+\rho \bfv$. This corresponds to a
 rotation by $120^{\circ}$ of the link tetrahedron about the axis
 passing through one vertex and the center of the opposite side. We
 have one such group of order 3 for every vertex of the tetrahedron,
 and they generate $A_4$.  The fundamental domain for the action of
 $\PGL_2(A)$ on the spine $X$ is a $30^{\circ} - 60^{\circ} -
 90^{\circ}$ triangle obtained by the barycentric subdivision of the
 hexagon.

 Let 
 \[ \alpha = b+c +d -2a
 \]
 \[ \beta=a+c+d-2b
 \]
 \[
  \gamma= a+b+d -2c 
  \]
 \[
  \delta = a+ b + c-2d 
 \] 
 where $a,b,c$ and $d$ are four integers, typically the values of a hermitian 
 form $f$ at four lax vectors parameterizing a vertex of $X$. 
 One easily verifies the following proposition: 
 
 \begin{prop}  \label{formula}
 Let 
 $f$  be an integral hermitian form. Let $\{\bfu,\bfv, \bfu+\bfv, \bfu+\rho \bfv\}$ be an 
 ultrabasis, and assume that $f$ takes values $a$, $b$, $c$ and $d$ at these four vectors. Then in the basis $(\bfu,\bfv)$ the form $f$ is written as 
  \[
  f(x,y) =\frac{1}{3}[\beta N(x) + \alpha N(y) + \gamma N(x-y) +\delta N(\rho x-y)] 
  \]
  and the discriminant of $f$ is 
 \[ 
\Delta= a^{2}+b^{2}+c^{2}+d^{2}-ab-ac-ad-bc-bd-cd=
-\frac{\alpha a+\beta b + \gamma c + \delta d}{2}. 
 \]
 \end{prop}

 The ultrabases $\{\bfu,\bfv, \bfu+\bfv, \bfu+\rho \bfv\}$ and 
 $\{\bfu,\bfv, \bfu+\rho \bfv, \bfu+\rho^{2} \bfv\}$  intersect in the superbasis 
 $\{\bfu,\bfv, \bfu+\rho \bfv\}$, thus the vertices corresponding to 
the two ultrabases are connected by an edge in the topograph (i.e. the
spine). 
 
 \begin{prop} \label{climbing} 
 We have the following for the adjacent vertices: 
 \[
 f(\bfu+\bfv) + f(\bfu+\rho^{2}\bfv)= f(\bfu)+f(\bfv)+f(\bfu+\rho \bfv). 
 \]
 \end{prop} 

 \begin{proof} 
Assume that the matrix of the form $f$ in the basis $(\bfu,\bfv)$ has integers
$a$, $c$ on the diagonal, and $\nu$, $\bar{\nu}$ off the diagonal.
Then the left hand side is $(a+c + tr(\nu))+(a+c+ tr(\rho^{2} \nu))$
while the right hand side is $a + c + (a+ c + tr(\rho \nu))$. The proposition follows 
since $1+\rho^{2}=\rho$. 
\end{proof}

Define $inv(v)$ as the sum of the values of $f$ on the 4 lax
vectors in the corresponding ultrabasis.  Any edge $e$ is determined
by a lax superbasis. Define $inv(e)$ as the sum of the values of $f$ on
the 3 lax vectors in the superbasis.  If $e$ connects $v$ and $v'$
then, by Proposition \ref{climbing}, 
\[
inv(v)+inv(v')=3 \cdot inv(e).
\]

Note that Proposition \ref{climbing} implies that $inv(v)\leq inv(v')$
for every vertex $v'$ adjacent to $v$ if and only if $\alpha, \beta, \gamma, 
\delta \geq 0$.  Following Conway, we shall call the vertex $v$ a {\em well}. 

\begin{prop} \label{well} Let $f$ be an integral 
 binary hermitian form.  Let $v$ be a vertex of $X$. Let $a,b,c$ and $d$ be the values of 
 $f$ at the 4 lax vectors in the corresponding ultrabasis.  Let $\alpha,\beta,\gamma$ and 
 $\delta$ be the integers defined by $a,b,c$ and $d$ as above. 
 The following are equivalent. 
 \begin{enumerate}
 \item The vertex $v$ is a well, i.e. $\alpha, \beta, \gamma, \delta \geq 0$. 
 \item If $\bfw$ is a lax vector that does not not belong to the ultrabasis of $v$, 
 then $f(\bfw)\geq a,b,c,d$. 
 \item For every vertex $u$,  $inv(v)\leq inv(u)$. 
 \end{enumerate} 
 In particular, a well exists if and only if $f$ is positive definite. The well is unique if 
 $\alpha, \beta, \gamma, \delta >0$.  
 \end{prop} 

\begin{proof}  Implications (2) $\Rightarrow$ (3) and (3) $\Rightarrow $(1) are trivial.
 We must show that (1) implies (2). Assume that the vertex $v$ corresponds to 
the ultrabasis $\{\bfu,\bfv, \bfu+\bfv, \bfu+\rho \bfv\}$. Write $w=x\bfu+y\bfv$, 
where $x,y\in A$. Since $\bfw$ different from $\bfu,\bfv, \bfu+\bfv$  and $\bfu+\rho \bfv$,
 the formula for $f$ in Proposition \ref{formula} implies that  
 \[
 f(x,y) \geq \frac{1}{3}(\alpha+\beta+\gamma+\delta) = a +\frac{\alpha}{3} 
 = b+\frac{\beta}{3} =c+ \frac{\gamma}{3} =d+\frac{\delta}{3}.
 \] 
 This implies that (2) holds if $v$ is a well. If $f$ is positive then the function $inv$ has 
 a minimum, so a well exists. Conversely, if there is a well, then $f$ is positive by the formula 
 in Proposition \ref{formula}. 
 If $\alpha, \beta, \gamma, \delta >0$ then $v$ is the unique 
 minimum of $inv$ so the well is unique. 
 \end{proof}

 One can use Proposition \ref{well} to classify positive definite
 forms of a fixed discriminant.  Take, for example, $\Delta =-2$. Then
 by Proposition \ref{formula}
 \[
 4= \alpha a + \beta b + \gamma c + \delta d\geq 
(\alpha+\beta+\gamma+\delta)s=(a+b+c+d)s 
 \] 
where $s$ is the smallest of the four values $a$, $b$, $c$ and $d$. 
It follows that $a=b=c=d=1$ at a well. Thus there is only one class 
 of forms of discriminant $-2$. 
 
\begin{thm} \label{T2}  Let $A=\bbZ[\rho]$, and let 
 $f$ be a primitive, integral, positive definite binary hermitian form
of discriminant $\Delta$. Then 
there exists $\bfv\neq 0$ in $A^2$ such that $f(\bfv)\leq \sqrt{-\frac{\Delta}{2}}$. 
Moreover, the equality is achieved only for the unique class of forms of 
discriminant $-2$. 
\end{thm} 

\begin{proof} 
Let $a,b,c$ and $d$ be the values of $f$ at a well. Let $s$ be the minimum 
of these values. Then 
\[
-2\Delta = \alpha a + \beta b + \gamma c + \delta d \geq 
(\alpha + \beta  + \gamma  + \delta )s=
(a+b+c+d)s \geq 4s^2. 
\]
with equalities if and only if $a=b=c=d$. Since $f$ is primitive, this forces
$a=b=c=d=1$, and $\Delta =-2$, as claimed. 
\end{proof}

Just as Conway's topograph, the spine $X$ can be defined combinatorially in terms of lax vectors, 
without any reference to $\mathbb H^3$. 
We shall now prove that $X$, considered simply as an affine cell complex, is contractible. 
The proof uses PL Morse Theory \cite{Be}. We start with the following lemma. 

\begin{lemma} \label{morse}
Let $f$ be an integral hermitian form of discriminant $\Delta$. 
Let $\sigma$ be a 2-dimensional cell in $X$. Identify $\sigma$ with the convex 
  hull of $\rho^i$, $i=0, \ldots ,5$ in $\mathbb C$. Then the function $inv$ can be extended 
  from the vertices of $\sigma$ to an affine function 
$inv: \mathbb C\rightarrow\mathbb R$. Moreover, 
if $3$ does not divide $\Delta$ then the affine function is not constant on any edge of 
$\sigma$. 
\end{lemma}
\begin{proof}
Assume that $\sigma$ corresponds to a lax basis $\{\bfu,\bfv\}$. Then the vertices $v_i$ 
of $\sigma$ correspond 
to ultrabases $\{\bfu,\bfv,\bfu+\rho^i\bfv, \bfu+\rho^{i+1}\bfv\}$, $i=0,\ldots ,5$. Assume that the matrix 
of $f$ in the basis $(\bfu,\bfv)$ has integers $a,c$ on the diagonal and $\nu,\bar\nu\in A^*$ 
off the diagonal. Then 
\[ 
inv(v_i)=3a+3c+ tr(\rho^i\nu) + tr(\rho^{i+1} \nu). 
\] 
Since $z\mapsto tr(z\nu)$ is a linear function on $\mathbb C$, it follows that $inv$ extends to 
an affine function on $\mathbb C$. If $inv(v_0)=inv(v_1)$ then $tr(\nu)=tr(\rho^2\nu)$ 
and this implies that $\nu$ is a real multiple of $\rho^2$. Since $\nu\in A^*$, $\nu$ must be 
an integral multiple of $\rho^2$. This implies that $N(\nu)$ is an integer, and the formula 
$\Delta=D(ac-N(\nu))$ implies that $D=-3$ divides $\Delta$. Thus there are no 
horizontal edges if $3$ does not divide $\Delta$.  
\end{proof} 

Let $f$ be a positive form with $\Delta=-2$. At the unique well $f$ has the value 1 at all 4 
lax vectors in the ultrabasis. 
Lemma \ref{morse} implies that $inv$  extends to a PL Morse function on 
the affine complex $X$, with the minimum at the well. 
Let $v$ be a vertex of $X$. Let $\sigma$ be a cell in $X$ that contains $v$ in its closure. 
We say that $\sigma$ is a descending cell if $inv(x)<inv(v)$ for all $x$ in the interior 
of $\sigma$. The descending link of $v$ is the link of $v$ in the union
of all descending cells. Now assume that $f$ has the values $a,b,c$
and $d$ at the 4 lax vectors in the ultrabasis. Let
$\alpha,\beta,\gamma$ and $\delta$ be defined by $a,b,c$ and $d$, as
before. By Lemma \ref{climbing} the descending edges correspond to
negative numbers amongst $\alpha,\beta,\gamma$ and $\delta$.  If we
order $a\geq b\geq c\geq d$, then it is clear that at least $\gamma$
and $\delta$ are positive.  Thus, if $v$ is not the well, we either
have one descending edge, or two descending edges. In the first case
the descending link is one point. In the second case, the hexagon
between two descending edges is also a descending cell, and the
descending link is a segment. In any case the descending link is
contractible. This proves that $X$ is contractible (see \cite{Be}).
 
 \smallskip 

If $f$ is indefinite and we are at a vertex where all values are 
positive then we can move to the ocean by 
passing to vertices with smaller average values. Note that along the way 
to the ocean we did not increase the minimum of $|f|$ at four values at 
each vertex. Thus the minimum of $|f|$ is achieved at the ocean $C$. 

Now assume that $v$ is 
a vertex of the ocean $C$. Then there are three possibilities: 

\begin{enumerate} 
\item three values at $v$ are positive and one is negative. 
\item two values at $v$ are positive and two are negative. 
\item one value at $v$ is positive and three are negative. 
\end{enumerate} 
In the first and third cases there are 3 hexagons in $C$ with $v$ as 
a common vertex, while in the second case there are 4 hexagons. 
Since regular hexagon has $120^{\circ}$ angle at any vertex, in the 
second case the sum of all angles at $v$ is $480^{\circ} > 360^{\circ}$ 
so $v$ is a point of negative curvature on $C$.  The following picture illustrates 
two possibilities at $v$. The thick edges of $Lk_X(v)$ correspond to 2-cells in $C$.

 \begin{picture}(300,180)(-80,-20)

\thicklines 

\put(00,40){\line(1,0){120}}
\put(00,40){\line(3,-1){90}}
\put(120,40){\line(-1,-1){30}}

\thinlines 

\put(00,40){\line(1,1){60}}
\put(120,40){\line(-1,1){60}}
\put(90,10){\line(-1,3){30}}

\put(50,50){$v$}
\put(60,50){\circle*{3}}

\put(00,90){$Lk_X(v)$}

\put(35,10){--}
\put(50,70){+}
\put(80,60){+}
\put(85,90){+} 

\thicklines

\put(160,40){\line(1,0){120}}
\put(160,40){\line(3,-1){90}}
\put(280,40){\line(-1,1){60}}
\put(250,10){\line(-1,3){30}}

\thinlines 

\put(280,40){\line(-1,-1){30}}
\put(160,40){\line(1,1){60}}

\put(210,50){$v$}
\put(220,50){\circle*{3}}

\put(160,90){$Lk_X(v)$}

\put(195,10){--}
\put(210,70){+}
\put(240,60){--}
\put(245,90){+} 

\end{picture}

We now show directly that there is a point of negative curvature on the ocean. 

\begin{prop}\label{negative} Let $f$ be an indefinite integral binary hermitian 
form, not representing 0. Then every connected component of an ocean $C$ contains a vertex 
where 4 hexagons meet, that is,  there are points of negative curvature in $C$. 
Moreover the minimum of $|f|$ is achieved at a vertex of 
negative curvature. 
\end{prop}
\begin{proof} Assume that we are at a vertex of $C$ corresponding to an 
ultrabasis $\{\bfu,\bfv,\bfu+\bfv,\bfu+\rho \bfv\}$ such that 
\[
f(\bfu+\bfv)\geq f(\bfu)\geq f(\bfu+\rho \bfv)> 0 > f(\bfv).
\]
Consider an adjacent vertex corresponding to the 
ultrabasis $\{\bfu,\bfv,\bfu+\rho \bfv, \bfu +\rho^{2} \bfv\}$, which clearly still 
belongs to $C$. Then 
 \[
 f(\bfu+\bfv) + f(\bfu+\rho^{2}\bfv)= f(\bfu)+f(\bfv)+f(\bfu+\rho\bfv)< 
 f(\bfu) + f(\bfu+\rho\bfv) 
 \]
and this implies that $f(\bfu+\rho^{2}\bfv)< f(\bfu+\rho\bfv)$. 
If $f(\bfu+\rho^{2}\bfv)<0$ we stop,
 otherwise we continue the process. Once two values are positive and two negative, 
 there are four hexagons in $C$ that meet at the vertex.  Also, note that this 
 process does not increase the minimum of four absolute values at the vertex. 
 Thus the minimum of $|f|$ is achieved at a vertex in $C$ of negative curvature. 
\end{proof}

We can use existence of negative curvature vertex to determine classes
of indefinite forms of fixed discriminant. Take $\Delta =6$, for
example.  Let $f$ be a form of discriminant $6$.  Assume that we are
at a vertex where $f$ takes values $a, b > 0 > c, d$.  Since $a^2+b^2
\geq 2ab$ and $c^2+d^2\geq 2cd$, the formula for the discriminant
implies that
\[ 
6 \geq ab - ac - ad - bc - bd +cd
\] 
where all summands on the right are positive. This is possible only if 
$a=b=1$ and $c=d=-1$.  Thus there is only one class of forms of 
discriminant $6$.

\begin{thm} \label{T3}  Let $A=\bbZ[\rho]$, and let 
 $f$ be a primitive, integral,  indefinite hermitian 
form of discriminant $\Delta$, not representing 0. Then 
there exists $\bfv\neq 0$ in $A^2$ such that $|f(\bfv)|\leq \sqrt{\frac{\Delta}{6}}$. 
Moreover, the equality is achieved only for the unique class of forms of 
discriminant $6$. 
\end{thm} 
\begin{proof} 
 Let $a,b,c$ and $d$ be the values of $f$ at a vertex of the ocean. 
By Proposition \ref{negative} we can assume that $a,b>0$ and $c,d<0$. 
Then $-ab$ and $-cd$ are the only negative terms in the expression for 
$\Delta$. Since $a^2+b^2\geq 2ab$ and $c^2+d^2\geq 2cd$, it follows that 
\[ 
\Delta \geq ab - ac - ad - bc - bd +cd\geq 6 s^2 
\] 
where $s$ is the minimum of $a,b, -c$ and $-d$. Note that equalities hold only
if $a=b=-c=-d$. Since $f$ is primitive, this implies that $a=b=1$ and 
$c=d=-1$. 
\end{proof}

\end{document}